\documentclass[oneside]{amsart}
\usepackage{amssymb}
\usepackage[a4paper]{geometry}
\usepackage{dsfont}
\usepackage{mathrsfs}
\usepackage{enumitem}
\usepackage[hidelinks]{hyperref}

\theoremstyle{plain}
\newtheorem{thm}{Theorem}[section] 
\newtheorem{prop}[thm]{Proposition} 
\newtheorem{lem}[thm]{Lemma} 
\newtheorem{cor}[thm]{Corollary}

\theoremstyle{remark} 
\newtheorem*{rem}{Remark}

\theoremstyle{definition}

\numberwithin{equation}{section}

\begin{document}

\title[Trace-free Beurling--Ahlfors and Bourgain--Brezis]{The trace-free Beurling--Ahlfors transform and the Bourgain--Brezis problem for Hodge systems}

\author{Diogo Ars\'enio}
\address{New York University Abu Dhabi \\
Abu Dhabi \\
United Arab Emirates} 
\email{\href{mailto:diogo.arsenio@nyu.edu}{diogo.arsenio@nyu.edu}}


\date{August 5, 2026}

\begin{abstract}	
	We show that the dual approach to Bourgain--Brezis estimates for Hodge systems is substantially more flexible than previously understood. For $1\leq l\leq n-1$, we introduce the trace-free Beurling--Ahlfors transform $S=\frac{n-l}{n}P-\frac lnP^\perp$, a canonical normalization of the generalized Beurling--Ahlfors transform on $l$-forms in $\mathbb{R}^n$. Its matrix symbol decomposes into scalar multipliers that are odd under suitable orthogonal reflections, yielding an endpoint cancellation estimate from finite measures to $L^\infty$ for $|D|^{-n}S$. This cancellation allows us to complete the Hilbertian case of the Bourgain--Brezis conjecture in every dimension and for every form degree. We then develop multilinear reflection estimates and obtain new critical Triebel--Lizorkin and Besov Bourgain--Brezis estimates. In particular, for every dimension and form degree, the Sobolev Bourgain--Brezis conjecture in $\dot W^{\frac np,p}$ holds for $p=\frac{2k}{2k-1}$, $k\geq1$, and hence for exponents arbitrarily close to $1$. We also derive endpoint Hodge decompositions and Hodge--Sobolev inequalities. Finally, except in the endpoint Besov case where the critical space already embeds into $L^\infty$, we prove that the associated bounded selections cannot be linear.
\end{abstract}

\maketitle

\tableofcontents


\section{Introduction and main results}\label{section:introduction}

The Bourgain--Brezis problem asks whether certain critical underdetermined equations admit bounded solutions despite the failure of the corresponding endpoint Sobolev embedding. Its simplest form concerns the divergence equation
\begin{equation}\label{eq:divergence:1}
	\operatorname{div} u=f
	\qquad \text{in } \mathbb R^n,
\end{equation}
where $n\geq2$. If $f\in L^n(\mathbb R^n)$, the canonical solution
\begin{equation*}
	u=-|D|^{-2}\nabla f
\end{equation*}
belongs to $\dot W^{1,n}(\mathbb R^n)$, but it need not belong to $L^\infty(\mathbb R^n)$. This reflects the failure of the critical Sobolev embedding
\begin{equation*}
	\dot W^{1,n}(\mathbb R^n)\not\subset L^\infty(\mathbb R^n).
\end{equation*}
Precise definitions of the homogeneous function spaces and operators used in this work are collected in Section~\ref{sec:preliminaries}.

Bourgain and Brezis nevertheless proved \cite{bourgain:brezis:2002,bourgain:brezis:2003} that one can choose another solution of \eqref{eq:divergence:1} satisfying
\begin{equation*}
	u\in \dot W^{1,n}(\mathbb R^n)\cap L^\infty(\mathbb R^n),
	\qquad
	\|u\|_{\dot W^{1,n}\cap L^\infty}
	\lesssim
	\|f\|_{L^n}.
\end{equation*}
The point is that the divergence equation is underdetermined. The freedom to add divergence-free vector fields can be used to compensate for the singularities of the canonical inverse. The corresponding choice of $u$, however, cannot be realized by a bounded linear selection operator.

Bourgain and Brezis subsequently extended this question to Hodge systems \cite{bourgain:brezis:2004,bourgain:brezis:2007}. Let $1\leq l\leq n-1$, and let $\Lambda^l=\Lambda^l\mathbb R^n$ denote the space of $l$-forms on $\mathbb R^n$. For a function space $X$, we write $X\Lambda^l$ for the corresponding space of $l$-forms with coefficients in $X$. If $v\in \dot W^{1,n}\Lambda^l$, then $dv\in L^n\Lambda^{l+1}$, and the Hodge-system version of the Bourgain--Brezis theorem asserts that
\begin{equation*}
	d\big[\dot W^{1,n}\Lambda^l\big]
	=
	d\big[(\dot W^{1,n}\cap L^\infty)\Lambda^l\big],
\end{equation*}
with equivalent induced norms. Here
\begin{equation*}
	\|dv\|_{d[X\Lambda^l]}
	=
	\inf_{du=dv}\|u\|_{X\Lambda^l}.
\end{equation*}
Equivalently, every exact datum $dv$ with a primitive $v$ in $\dot W^{1,n}\Lambda^l$ has another primitive $u$ belonging simultaneously to $\dot W^{1,n}\Lambda^l$ and $L^\infty\Lambda^l$.

The general Bourgain--Brezis conjecture asks whether the same phenomenon holds throughout the critical Sobolev scale. Specifically, in \cite{bourgain:brezis:2007}, they conjectured that
\begin{equation}\label{conjecture:1}
	d\big[\dot W^{\frac np,p}\Lambda^l\big]
	=
	d\big[(\dot W^{\frac np,p}\cap L^\infty)\Lambda^l\big],
	\qquad
	1<p<\infty,\quad 1\leq l\leq n-1,
\end{equation}
again with equivalent induced norms. In concrete terms, this asks whether for every $v\in \dot W^{\frac np,p}\Lambda^l$ one can find $u\in(\dot W^{\frac np,p}\cap L^\infty)\Lambda^l$ such that $du=dv$ and
\begin{equation*}
	\|u\|_{(\dot W^{\frac np,p}\cap L^\infty)\Lambda^l}
	\lesssim
	\|v\|_{\dot W^{\frac np,p}\Lambda^l}.
\end{equation*}
The case $p=n$ is the original Bourgain--Brezis theorem. Beyond $p=n$, progress has been limited to special exponents, form degrees, or regions of the parameter space.

Several important partial results are known. Bourgain and Brezis observed that, in dimension $2$, the Hilbertian case admits a simpler proof based on duality. Maz'ya extended this dual approach in higher dimensions, proving the case $(p,l)=(2,n-1)$ in \cite{maz'ya:2007} and the case $(p,l)=(2,1)$ in \cite{maz'ya:2010}. These results showed that duality is a powerful tool in the problem. In all known applications, however, the dual method remained confined to the Hilbertian setting $p=2$.

A different line of progress was developed by Bousquet, Mironescu, and Russ \cite{bousquet:mironescu:russ:2013}, and later by Bousquet, Russ, Wang, and Yung \cite{bousquet:russ:wang:yung:2019}. Their work refines the original Bourgain--Brezis approximation method and proves estimates in the Triebel--Lizorkin scale. The following formulation summarizes the relevant part of the known theory.

\begin{thm}[\cite{bousquet:russ:wang:yung:2019,maz'ya:2010}]\label{thm:1}
	Let $n\geq2$, $1\leq l\leq n-1$, and $1<p,q<\infty$. Assume
	that
	\begin{equation*}
		n<p+l
		\qquad\text{or}\qquad
		(p,q,l)=(2,2,1).
	\end{equation*}
	Then
	\begin{equation*}
		d\big[\dot F^{\frac np}_{p,q}\Lambda^l\big]
		=
		d\big[(\dot F^{\frac np}_{p,q}\cap L^\infty)\Lambda^l\big],
	\end{equation*}
	with equivalent induced norms.
\end{thm}

The condition $n<p+l$ is well adapted to large values of $p$. In particular, it contains the original Sobolev case $p=n$, and it covers the full range $1<p<\infty$ when $l=n-1$. However, if $l<n-1$, then, apart from the exceptional Hilbertian case $(p,q,l)=(2,2,1)$, Theorem~\ref{thm:1} leaves out the interval $1<p\leq 2$. Thus, away from the highest form degree, the Bourgain--Brezis conjecture remained largely inaccessible near the
endpoint $p=1$.

\medskip

Bourgain--Brezis estimates have also found applications to phase lifting and Ginzburg--Landau problems, wave and fluid systems, and rigidity estimates in elasticity. See \cite{bourgain:brezis:2003,bourgain:brezis:2007,chanillo:yung:2012, chanillo:vanschaftingen:yung:2016,conti:garroni:2021}, and \cite{vanschaftingen:2014} for a broader account.

\medskip

The purpose of this paper is to show that the dual method is substantially more flexible than previously understood. The novelty is not the use of duality itself. Rather, we identify within the dual formulation a trace-free Beurling--Ahlfors structure that is valid in every dimension and for every form degree. This intrinsic structure completes the Hilbertian theory, and its underlying cancellation properties admit multilinear extensions that reach new non-Hilbertian families of exponents, including values of $p>1$ arbitrarily close to $1$.

In order to describe the trace-free Beurling--Ahlfors structure and its consequences in more detail, let us consider
\begin{equation*}
	P=\frac{dd^*}{|D|^2}
	\qquad\text{and}\qquad
	P^\perp=\frac{d^*d}{|D|^2},
\end{equation*}
the Hodge projections onto exact and coexact $l$-forms.
The operator
\begin{equation*}
	\widetilde S=P-P^\perp
\end{equation*}
is the generalized Beurling--Ahlfors transform on differential forms \cite{iwaniec:martin:1991, iwaniec:1992, iwaniec:martin:1993,iwaniec:martin:1996}.

We write $P(\xi)$ and $P^\perp(\xi)$ for the corresponding matrix-valued Fourier symbols acting on $\Lambda^l$. For every $\xi\neq0$, as detailed later in \eqref{eq:hodge_ranks}, one has
\begin{equation*}
	\operatorname{rank}P(\xi)=\binom{n-1}{l-1},
	\qquad
	\operatorname{rank}P^\perp(\xi)=\binom{n-1}{l}.
\end{equation*}
Since $P(\xi)$ and $P^\perp(\xi)$ are idempotent, their traces equal their ranks. Therefore,
\begin{equation*}
	\operatorname{tr}\widetilde S(\xi)
	=
	\binom{n-1}{l-1}-\binom{n-1}{l}
	=
	\frac{2l-n}{n}\binom nl
	=
	\frac{2l-n}{n}\dim\Lambda^l.
\end{equation*}

Next, we introduce the normalized trace-free Beurling--Ahlfors transform
\begin{equation}\label{eq:intro:S}
	S
	=
	\frac12\left(
	\widetilde S-\frac{2l-n}{n}\operatorname{Id}
	\right)
	=
	\frac{n-l}{n}P-\frac lnP^\perp.
\end{equation}
Thus,
\begin{equation*}
	\operatorname{tr}S(\xi)=0
	\qquad
	\text{for every $\xi\neq0$}.
\end{equation*}

In coordinates, we will show that every coefficient of $S\omega$ is a finite linear combination of scalar second-order Riesz transforms acting on the coefficients of $\omega$, with symbols of the forms
\begin{equation*}
	\frac{\xi_i^2-\xi_j^2}{|\xi|^2}
	\qquad\text{and}\qquad
	\frac{\xi_i\xi_j}{|\xi|^2},
\end{equation*}
with $i\neq j$. Each such symbol is odd under a suitable orthogonal reflection. The use of oddness symmetries in a dual Bourgain--Brezis argument already appears in the two-dimensional divergence setting in \cite[Section~4]{bourgain:brezis:2003}. The new point here is that the operator $S$ organizes the components of the Hodge decomposition into multipliers with this symmetry, in every dimension and for every form degree. Moreover, this organization persists in the multilinear estimates developed below.

The transform $S$ provides an intrinsic framework for the dual Bourgain--Brezis problem. Its fundamental linear endpoint estimate is the following.

We denote by $\mathscr M(\mathbb R^n)$ the space of finite signed Radon measures, equipped with the total variation norm, by $C(\mathbb R^n)$ the space of bounded continuous functions, equipped with the uniform norm, by $C_0(\mathbb R^n)$ its subspace of functions vanishing at infinity, and by $L^1_0(\mathbb R^n)$ the mean-zero subspace of $L^1(\mathbb R^n)$.

\begin{thm}[Endpoint cancellation]\label{thm:intro:beurling}
	Let $n\geq2$ and $1\leq l\leq n-1$. For every finite
	$\Lambda^l$-valued Radon measure $\mu$ on $\mathbb R^n$,
	\begin{equation*}
	\big\||D|^{-n}S\mu\big\|_{L^\infty\Lambda^l}
	\lesssim
	\|\mu\|_{\mathscr M \Lambda^l}.
	\end{equation*}
	Moreover,
	\begin{equation*}
	|D|^{-n}S\big[L^1\Lambda^l\big]\subset C\Lambda^l,
	\qquad
	|D|^{-n}S\big[L^1_0\Lambda^l\big]\subset C_0\Lambda^l.
	\end{equation*}
\end{thm}

The estimate in Theorem~\ref{thm:intro:beurling} fails if $S$ is replaced by the identity. The normalization in \eqref{eq:intro:S} removes the scalar multiple of the identity from the matrix-valued symbol of $\widetilde S$. As previously mentioned, the remaining trace-free symbol decomposes into reflection-odd multipliers, and it is this cancellation that produces the boundedness from measures to $L^\infty$. Theorem~\ref{thm:intro:beurling} is the linear manifestation of a mechanism that will later be extended to multilinear expressions.

The first consequence is the full Hilbertian case of the Bourgain--Brezis conjecture.

\begin{thm}[Hilbertian Bourgain--Brezis estimates]\label{thm:intro:hilbertian}
	Let $n\geq2$ and $1\leq l\leq n-1$. Then
	\begin{equation*}
		d\big[\dot H^{\frac n2}\Lambda^l\big]
		=
		d\big[(\dot H^{\frac n2}\cap L^\infty)\Lambda^l\big],
	\end{equation*}
	with equivalent induced norms.
\end{thm}

This completes the $p=2$ case of \eqref{conjecture:1} for all dimensions and all form degrees. Later in the paper, this is obtained as a consequence of the dual estimate
\begin{equation*}
	\|\omega\|_{\dot H^{-\frac n2}\Lambda^l}
	\lesssim
	\|\omega\|_{L^1_0\Lambda^l}
	+
	\|P\omega\|_{\dot H^{-\frac n2}\Lambda^l}.
\end{equation*}
The role of $S$ appears through the identity
\begin{equation*}
	\frac{n-l}{n}\|P\omega\|_{\dot H^{-\frac n2}\Lambda^l}^2
	-
	\frac l n \|P^\perp\omega\|_{\dot H^{-\frac n2}\Lambda^l}^2
	=
	\int_{\mathbb R^n}
	\big\langle \omega, |D|^{-n}S\omega\big\rangle dx,
\end{equation*}
which, through the endpoint boundedness of $|D|^{-n}S$, allows the prescribed control of the exact component $P\omega$ to be transferred to the coexact component $P^\perp \omega$.

The reflection cancellation underlying Theorem~\ref{thm:intro:beurling} can be multilinearized. In the Triebel--Lizorkin scale, this yields the principal new range for the Bourgain--Brezis conjecture. For $1<r<\infty$, we write $r'=\frac r{r-1}$ for the conjugate exponent.

\begin{thm}[Triebel--Lizorkin estimates]\label{thm:2}
	Let $n\geq2$, $1\leq l\leq n-1$, and $1<p\leq q\leq2$. Assume that
	\begin{equation*}
		q' \text{ is even}
		\qquad\text{and}\qquad
		\frac{p'}{q'}\in\mathbb N.
	\end{equation*}
	Then
	\begin{equation*}
		d\big[\dot F^{\frac np}_{p,q}\Lambda^l\big]
		=
		d\big[(\dot F^{\frac np}_{p,q}\cap L^\infty)\Lambda^l\big],
	\end{equation*}
	with equivalent induced norms.
\end{thm}

The proof proceeds by duality and replaces the linear endpoint estimate underlying Theorem~\ref{thm:intro:beurling} by a multilinear one. The evenness of $q'$ permits the expansion of the relevant Littlewood--Paley powers, while the condition $p'/q'\in\mathbb N$ permits the expansion of the remaining outer power after the Littlewood--Paley sum has been formed. The resulting terms are controlled by the same reflection-odd multipliers as in the linear estimate.

Taking $q=2$ and using the identification $\dot W^{s,p}=\dot F^s_{p,2}$, with equivalent norms, Theorem~\ref{thm:2} gives Sobolev Bourgain--Brezis estimates for
\begin{equation*}
	p=\frac{2k}{2k-1},
	\qquad
	k=1,2,\ldots
\end{equation*}
for every $n\geq2$ and every $1\leq l\leq n-1$. Hence, for each dimension and form degree, the conjecture holds along a sequence of exponents converging to $1$.

A parallel multilinear argument yields the following result in the Besov scale.

\begin{thm}[Besov estimates]\label{thm:intro:besov}
	Let $n\geq2$, $1\leq l\leq n-1$, $1<p\leq 2$ and $1\leq q\leq p$.
	Assume that $p'$ is an even integer. Then
	\begin{equation*}
		d\big[\dot B^{\frac np}_{p,q}\Lambda^l\big]
		=
		d\big[(\dot B^{\frac np}_{p,q}\cap L^\infty)\Lambda^l\big],
	\end{equation*}
	with equivalent induced norms.
\end{thm}

Theorems~\ref{thm:1} and~\ref{thm:2} cover complementary regions of the critical parameter range. The Bourgain--Brezis approximation method is strongest for large $p$, through the condition $n<p+l$, whereas the method developed here reaches exponents $p$ arbitrarily close to $1$. This complementarity raises a natural interpolation problem. Any such interpolation argument would need to account for the intrinsic nonlinearity of Bourgain--Brezis selections, to which we return in the final section. Related complex and real interpolation approaches have recently been
developed in \cite{curca:2025,curca:2026}.

\medskip

The paper is organized as follows. Section~\ref{sec:preliminaries} collects the notation for homogeneous Littlewood--Paley spaces, differential forms, and Hodge projections. Section~\ref{sec:duality} recalls the duality principle connecting Bourgain--Brezis selections with estimates for Hodge components of integrable forms. In Section~\ref{section:beurling}, we introduce the trace-free Beurling--Ahlfors transform and establish its endpoint reflection cancellation. Section~\ref{section:hilbertian} applies this estimate to the Hilbertian problem and its dyadic Besov refinement. The multilinear reflection estimates and the resulting Triebel--Lizorkin and Besov theorems are proved in Section~\ref{section:multilinear}. Section~\ref{section:hodge_consequences} develops consequences for endpoint Hodge decompositions and Hodge--Sobolev estimates. Section~\ref{section:nonlinearity} proves the intrinsic nonlinearity of the associated bounded selections. Appendix~\ref{appendix:one-dimensional-model} presents a one-dimensional model based on the Hilbert transform and Hardy projections, while Appendix~\ref{appendix:further_properties} records supplementary properties of the trace-free Beurling--Ahlfors transform.


\section{Preliminaries}\label{sec:preliminaries}

For clarity, we collect here the notation and some standard functional-analytic facts used throughout the paper.

Generally speaking, we employ the expression $A\lesssim B$ to denote an inequality $A\leq CB$ for a constant $C>0$ depending only on fixed parameters, and $A\sim B$ if both $A\lesssim B$ and $B\lesssim A$ hold.

\subsection{Homogeneous functional spaces}

Let $\mathscr S(\mathbb R^n)$ and $\mathscr S'(\mathbb R^n)$ denote the Schwartz space and the space of tempered distributions, respectively, and let $\mathscr P(\mathbb R^n)$ denote the space of polynomials.

All homogeneous function spaces in this paper are understood as subspaces of $\mathscr S'(\mathbb R^n)/\mathscr P(\mathbb R^n)$, the space of tempered distributions modulo polynomials.
We let $\mathscr S_0(\mathbb R^n)$ denote the subspace of all $\phi\in\mathscr S(\mathbb R^n)$ satisfying
\begin{equation*}
	\int_{\mathbb R^n}x^\alpha\phi(x)dx=0
\end{equation*}
for every multi-index $\alpha$. It can be shown that the quotient space $\mathscr{S}'/\mathscr{P}$ is isomorphic to the space $\mathscr{S}_0'(\mathbb{R}^n)$ of tempered distributions restricted to $\mathscr{S}_0$.
We follow the framework of \cite{triebel:1983} and also refer to \cite{grafakos:modern:2014} for further details on the subject.

Fix now a nonnegative radial function $\varphi\in C^\infty_c(\mathbb R^n\setminus\{0\})$ such that
\begin{equation*}
	\sum_{j\in\mathbb Z}\varphi(2^{-j}\xi)=1
	\qquad
	\text{for every $\xi\neq0$},
\end{equation*}
and set
\begin{equation*}
	\Delta_j=\varphi(2^{-j}D).
\end{equation*}
Besides the Lebesgue spaces $L^p(\mathbb{R}^n)$, with $1\leq p\leq\infty$, we will use Triebel--Lizorkin spaces $\dot F^s_{p,q}(\mathbb{R}^n)$, with $(s,p,q)\in\mathbb{R}\times[1,\infty)\times[1,\infty]$, and Besov spaces $\dot B^s_{p,q}(\mathbb{R}^n)$, with $(s,p,q)\in\mathbb{R}\times[1,\infty]\times[1,\infty]$. These are defined via a smooth homogeneous Littlewood--Paley decomposition
\begin{equation*}
	u=\sum_{j\in\mathbb{Z}}\Delta_{j}u
\end{equation*}
as the subspaces of $\mathscr{S}'/\mathscr{P}$ endowed with the complete norms
\begin{equation*}
	\left\|u\right\|_{\dot F^{s}_{p,q}\left(\mathbb{R}^n\right)}=
	\left\|2^{js}\Delta_j u\right\|_{L^p \ell^q}
	\qquad\text{and}\qquad
	\left\|u\right\|_{\dot B^{s}_{p,q}\left(\mathbb{R}^n\right)}=
	\left\|2^{js}\Delta_j u\right\|_{\ell^q L^p},
\end{equation*}
where the $\ell^q$-norm acts on $j\in\mathbb{Z}$.

We introduce
\begin{equation*}
	L^1_0(\mathbb R^n)
	=
	\bigg\{f\in L^1(\mathbb R^n):
	\int_{\mathbb R^n}f(x)dx=0\bigg\}
\end{equation*}
to denote the restriction of $L^1(\mathbb{R}^n)$ to mean-zero functions. While $\mathscr{S}_0$ is dense in $\dot F^{s}_{p,q}$ and $\dot B^{s}_{p,q}$, for any $s\in\mathbb{R}$ and $p,q\in [1,\infty)$, it is important to recall here that the closure of $\mathscr{S}_0$ in $L^1$ is $L^1_0$. Moreover,
\begin{equation*}
	(L^1_0)'=L^\infty/\mathbb R.
\end{equation*}
Since constants vanish in $\mathscr{S}'/\mathscr{P}$, we identify this quotient with $L^\infty$ whenever we work in the homogeneous setting.

When $s\in\mathbb{N}$ and $1\leq p\leq\infty$, we will also use the classical homogeneous Sobolev space $\dot W^{s,p}(\mathbb{R}^n)$ of tempered distributions $f\in\mathscr{S}'/\mathscr{P}$ such that $\partial^\alpha f\in L^p(\mathbb{R}^n)$, for all multi-indices $\alpha\in \mathbb{N}^n$ with $|\alpha|=s$. If $1<p<\infty$ and $s\in\mathbb{R}$, we will employ the same notation $\dot W^{s,p}(\mathbb{R}^n)$ to refer to the homogeneous Sobolev space defined by the condition $|D|^sf\in L^p(\mathbb{R}^n)$. If $p=2$, we will also use the classical notation $\dot W^{s,2}=\dot H^s$. Note that the preceding definitions of $\dot W^{s,p}(\mathbb{R}^n)$ are equivalent when $s\in\mathbb{N}$ and $1< p<\infty$.

For every $s\in\mathbb R$ and $1<p<\infty$, one has the standard identification
\begin{equation*}
	\dot W^{s,p} = \dot F^s_{p,2},
\end{equation*}
with equivalent norms.

We also recall here that, for $1<p,q<\infty$ and $s\in\mathbb{R}$, the standard duality identities are
\begin{equation*}
	\big(\dot F^s_{p,q}\big)'
	=
	\dot F^{-s}_{p',q'},
	\qquad
	\big(\dot B^s_{p,q}\big)'
	=
	\dot B^{-s}_{p',q'},
\end{equation*}
with equivalent norms. The duality pairings are the extensions of the usual integral pairing on $\mathscr{S}_0$.

Finally, if $X$ and $Y$ are compatible normed spaces, in the sense that they are both subspaces of a common Hausdorff topological vector space, we equip their intersection and sum with the norms
\begin{equation*}
	\|u\|_{X\cap Y} = \max\big\{\|u\|_X,\|u\|_Y\big\}
\end{equation*}
and
\begin{equation*}
	\|u\|_{X+Y} = \inf_{u=u_0+u_1} \big( \|u_0\|_X+\|u_1\|_Y \big),
\end{equation*}
respectively.

\subsection{Differential forms}

For $0\leq l\leq n$, we write $\Lambda^l=\Lambda^l\mathbb{R}^n$ to denote the space of exterior $l$-forms on the Euclidean space $\mathbb{R}^n$. We use the convention that $\Lambda^0=\mathbb{R}$ and $\Lambda^l=\{0\}$ whenever $l<0$ or $l>n$.
Each $l$-form has the unique representation
\begin{equation*}
	\omega=\sum_{|I|=l}\omega_I dx_I,
\end{equation*}
where the summation runs over all ordered $l$-tuples
\begin{equation*}
	I=(i_1<i_2<\ldots<i_l),
\end{equation*}
with $\{i_1,i_2,\ldots,i_l\}\subset\{1,2,\ldots,n\}$, and the elements
\begin{equation*}
	dx_I=dx_{i_1}\wedge dx_{i_2}\wedge\ldots\wedge dx_{i_l}
\end{equation*}
constitute an orthonormal basis of $\Lambda^l$.
We write
\begin{equation*}
	\langle\omega,\eta\rangle
	=
	\sum_{|I|=l}\omega_I\eta_I
\end{equation*}
for the Euclidean inner product on $\Lambda^l$.

If $X$ is a normed space of scalar functions or distributions, then $X(\mathbb R^n,\Lambda^l)$, abbreviated to $X\Lambda^l$, denotes the space of differential $l$-forms whose coefficients belong to $X$, equipped with the norm
\begin{equation*}
	\|\omega\|_{X\Lambda^l}=\bigg(\sum_{|I|=l}\|\omega_I\|_X^2\bigg)^\frac 12.
\end{equation*}
Since $\Lambda^l$ is finite dimensional, the precise choice of a norm on the coefficient space is immaterial up to equivalence. All boundedness, density, and duality statements used below extend coefficientwise from $X$ to $X\Lambda^l$. We refer to \cite{abraham:marsden:ratiu:1988,schwarz:1995} for background on differential forms.

The exterior derivative on $l$-forms is given by
\begin{equation*}
	d\omega
	=
	\sum_{j=1}^n\sum_{|I|=l}
	\partial_j\omega_I dx_j\wedge dx_I,
\end{equation*}
and its formal $L^2$-adjoint is the codifferential $d^*$. On $l$-forms, it can be written as
\begin{equation*}
	d^*=(-1)^{n(l-1)+1}\star d\star,
\end{equation*}
where $\star$ denotes the Hodge star operator. These operators satisfy
\begin{equation*}
	d^2=0,
	\qquad
	(d^*)^2=0,
	\qquad
	dd^*+d^*d=-\Delta,
\end{equation*}
where $\Delta=\sum_{j=1}^n\partial_j^2$ is the classical Laplacian and acts coefficientwise.

The exterior derivative and codifferential act boundedly as
\begin{equation*}
	d:\dot W^{s,p}\Lambda^l\to \dot W^{s-1,p}\Lambda^{l+1}
\end{equation*}
and
\begin{equation*}
	d^*:\dot W^{s,p}\Lambda^l\to \dot W^{s-1,p}\Lambda^{l-1}.
\end{equation*}
Similar boundedness properties hold in the scales of Triebel--Lizorkin and Besov spaces.

\subsection{Hodge projections}\label{subsec:hodge_projections}

The classical Riesz transforms are defined by
\begin{equation*}
	R_j=\frac{\partial_j}{|D|}=\mathscr{F}^{-1}\frac{i\xi_j}{|\xi|}\mathscr{F},
	\qquad
	j=1,\ldots,n.
\end{equation*}
Here, we use the convention
\begin{equation*}
	\mathscr{F}f(\xi)=\widehat f(\xi)=\int_{\mathbb{R}^n}e^{-i\xi\cdot x}f(x)dx
\end{equation*}
defining the Fourier transform.

On differential forms, we use the exterior and codifferential Riesz transforms
\begin{equation*}
	R=\frac d{|D|},
	\qquad
	R^*=\frac{d^*}{|D|}.
\end{equation*}
The Hodge projections are
\begin{equation*}
	P=RR^*=\frac{dd^*}{|D|^2},
	\qquad
	P^\perp=R^*R=\frac{d^*d}{|D|^2}.
\end{equation*}
They satisfy
\begin{equation*}
	P^2=P,
	\qquad
	(P^\perp)^2=P^\perp,
	\qquad
	PP^\perp=P^\perp P=0,
	\qquad
	P+P^\perp=\operatorname{Id}.
\end{equation*}
In particular, in $\mathscr S_0'\Lambda^l$, one has
\begin{equation*}
	\begin{aligned}
		P^\perp\omega&=0
		&&\Leftrightarrow&
		d\omega&=0,
		\\
		P\omega&=0
		&&\Leftrightarrow&
		d^*\omega&=0.
	\end{aligned}
\end{equation*}
Thus $P$ and $P^\perp$ are the $L^2$-orthogonal projections onto the exact and coexact components, respectively.

The geometry of the symbols will be used repeatedly. We identify $\xi\in\mathbb R^n$ with its Euclidean dual vector, and let $\varepsilon_\xi\omega=\xi\wedge\omega$ and $\iota_\xi\omega=\omega(\xi,\ldots)$ denote exterior and interior multiplication, respectively. With the Fourier transform convention above, one has
\begin{equation*}
	\widehat{d\omega}(\xi)
	=i\varepsilon_\xi\widehat\omega(\xi),
	\qquad
	\widehat{d^*\omega}(\xi)
	=-i\iota_\xi\widehat\omega(\xi).
\end{equation*}
Consequently, the symbols of the Hodge projections are
\begin{equation*}
	P(\xi)
	=
	\frac{\varepsilon_\xi\iota_\xi}{|\xi|^2},
	\qquad
	P^\perp(\xi)
	=
	\frac{\iota_\xi\varepsilon_\xi}{|\xi|^2},
	\qquad
	\xi\neq0.
\end{equation*}
The identity
\begin{equation*}
	\varepsilon_\xi\iota_\xi
	+
	\iota_\xi\varepsilon_\xi
	=
	|\xi|^2\operatorname{Id}
\end{equation*}
is the symbol-level form of the Hodge identity.

Let $\xi\neq 0$ and set $e=\xi/|\xi|$. Choosing an orthonormal basis of $\mathbb R^n$ whose first element is $e$, every basis element of $\Lambda^l$ either contains $e$ or belongs entirely to $e^\perp$. Hence one has the orthogonal decomposition
\begin{equation*}
	\Lambda^l
	=
	e\wedge\Lambda^{l-1}(e^\perp)
	\oplus
	\Lambda^l(e^\perp).
\end{equation*}

The symbol $P(\xi)$ is the projection onto the first summand, whereas $P^\perp(\xi)$ is the projection onto the second. Consequently,
\begin{equation}\label{eq:hodge_ranks}
	\operatorname{rank}P(\xi)
	=\binom{n-1}{l-1},
	\qquad
	\operatorname{rank}P^\perp(\xi)
	=\binom{n-1}{l}.
\end{equation}

Since the coefficients of $P$ and $P^\perp$ are finite linear combinations of second-order Riesz transforms, both projections are bounded on $\dot F^s_{p,q}\Lambda^l$ for $s\in\mathbb R$ and $1<p,q<\infty$, and on $\dot B^s_{p,q}\Lambda^l$ for $s\in\mathbb R$, $1<p<\infty$, and $1\leq q\leq\infty$. In either functional setting one therefore has
\begin{equation*}
	\omega=P\omega+P^\perp\omega
\end{equation*}
and
\begin{equation*}
	\|\omega\|_{X\Lambda^l}
	\sim
	\|P\omega\|_{X\Lambda^l}
	+
	\|P^\perp\omega\|_{X\Lambda^l},
\end{equation*}
where $X$ denotes any of the preceding spaces in the boundedness ranges of $P$ and $P^\perp$.
Moreover,
\begin{equation*}
	P\big[\dot F^s_{p,q}\Lambda^l\big]
	=
	d\big[\dot F^{s+1}_{p,q}\Lambda^{l-1}\big],
	\qquad
	P^\perp\big[\dot F^s_{p,q}\Lambda^l\big]
	=
	d^*\big[\dot F^{s+1}_{p,q}\Lambda^{l+1}\big],
\end{equation*}
and analogous identities hold in the Besov scale. Hence
\begin{equation*}
	\dot F^s_{p,q}\Lambda^l
	=
	d\big[\dot F^{s+1}_{p,q}\Lambda^{l-1}\big]
	\oplus
	d^*\big[\dot F^{s+1}_{p,q}\Lambda^{l+1}\big],
\end{equation*}
with a corresponding decomposition for $\dot B^s_{p,q}\Lambda^l$. These are the homogeneous Hodge decompositions in the Triebel--Lizorkin and Besov scales.


\section{The duality principle}\label{sec:duality}

A dual formulation of Bourgain--Brezis estimates was already present in the original work of Bourgain and Brezis \cite{bourgain:brezis:2003,bourgain:brezis:2007}. We record here the precise form that will be used throughout the paper. The point of the proposition below is to identify the Bourgain--Brezis selection estimate with an endpoint estimate for the coexact Hodge component of an integrable form.

Let $1<p,q<\infty$, and let $Y$ denote either
\begin{equation*}
	\dot F^{\frac np}_{p,q}(\mathbb R^n)
	\qquad\text{or}\qquad
	\dot B^{\frac np}_{p,q}(\mathbb R^n).
\end{equation*}
We write $X=Y'$, so that, respectively,
\begin{equation*}
	X=\dot F^{-\frac np}_{p',q'}(\mathbb R^n)
	\qquad\text{or}\qquad
	X=\dot B^{-\frac np}_{p',q'}(\mathbb R^n).
\end{equation*}
Since $1<p,q<\infty$, the space $Y$ is reflexive, and hence $X'=Y$.

With the standard norms on sums and intersections, one has
\begin{equation}\label{eq:dual-sum-intersection}
	\big(X+L^1_0\big)'=Y\cap L^\infty,
\end{equation}
where, as usual in the homogeneous setting, $L^\infty$ is understood modulo constants. All duality pairings below are taken coefficientwise on differential forms.

We first record why the mean-zero space $L^1_0$ occurs naturally.

\begin{lem}\label{lemma:critical-mean-zero}
	Let $X$ be either of the two negative smoothness spaces above. If $f\in X\cap L^1(\mathbb R^n)$, then
	\begin{equation*}
		\int_{\mathbb R^n}f(x)dx=0.
	\end{equation*}
\end{lem}

\begin{proof}
	Set $r=\max\{p',q'\}$. The standard critical Sobolev embeddings \cite{triebel:1983} give
	\begin{equation*}
		X\subset
		\dot F^{-n(1-\frac1r)}_{r,r}
		=
		\dot B^{-n(1-\frac1r)}_{r,r}.
	\end{equation*}
	Consequently, if $f\in X$, then
	\begin{equation}\label{eq:low-frequency-vanishing}
		\lim_{j\to-\infty}
		2^{-jn(1-\frac1r)}\|\Delta_jf\|_{L^r}=0.
	\end{equation}
	
	Next, write $\Delta_jf=\mathcal K_j*f$, where
	\begin{equation*}
		\mathcal K_j(x)=2^{jn}\mathcal K(2^jx)
	\end{equation*}
	for a nonzero kernel $\mathcal K\in\mathscr S_0$. A change of variables
	gives
	\begin{equation*}
		2^{-jn(1-\frac1r)}\|\Delta_jf\|_{L^r}
		=
		\left\|
		\int_{\mathbb R^n}
		\mathcal K(x-2^jy)f(y)dy
		\right\|_{L^r}.
	\end{equation*}
	Further observe that
	\begin{equation*}
		\left\|\int_{\mathbb R^n}
		\big(
		\mathcal K(x-2^jy)-\mathcal K(x)
		\big)f(y)dy
		\right\|_{L^r}
		\leq
		\int_{\mathbb R^n}
		\|\mathcal K(x-2^jy)-\mathcal K(x)\|_{L^r}
		|f(y)|dy.
	\end{equation*}
	Since translations are continuous in $L^r$, the integrand on the right-hand side converges pointwise to zero as $j\to-\infty$. Moreover, it is bounded by $2\|\mathcal K\|_{L^r}|f(y)|$, which yields, by dominated convergence, that
	\begin{equation*}
		\int_{\mathbb R^n}
		\mathcal K(x-2^jy)f(y)dy
		\rightarrow
		\mathcal K(x)
		\int_{\mathbb R^n}f(y)dy
	\end{equation*}
	in $L^r$. Therefore,
	\begin{equation*}
		\lim_{j\to-\infty}
		2^{-jn(1-\frac1r)}\|\Delta_jf\|_{L^r}
		=
		\|\mathcal K\|_{L^r}
		\left|\int_{\mathbb R^n}f(y)dy\right|.
	\end{equation*}
	The conclusion then follows from \eqref{eq:low-frequency-vanishing} and the fact that $\mathcal K$ is nonzero.
\end{proof}

The preceding lemma clarifies why the statement of the duality principle below is formulated in terms of integrable differential forms with mean-zero coefficients.

\begin{prop}[Duality principle]\label{prop:duality}
	Let $n\geq2$, $1\leq l\leq n-1$, and let $Y$ and $X=Y'$ be as above.
	The following assertions are equivalent.
	\begin{enumerate}[label=(\Alph*),font=\upshape]
		\item\label{estimate:equiv1}
		For every $v\in Y\Lambda^l$, there exists
		$u\in(Y\cap L^\infty)\Lambda^l$ such that $du=dv$ and
		\begin{equation*}
			\|u\|_{(Y\cap L^\infty)\Lambda^l}
			\lesssim
			\|v\|_{Y\Lambda^l}.
		\end{equation*}
		
		\item\label{estimate:equiv2}
		Whenever $\omega\in(X+L^1_0)\Lambda^l$ and
		$P\omega\in X\Lambda^l$, one has $\omega\in X\Lambda^l$ and
		\begin{equation*}
			\|\omega\|_{X\Lambda^l}
			\lesssim
			\|\omega\|_{(X+L^1_0)\Lambda^l}
			+
			\|P\omega\|_{X\Lambda^l}.
		\end{equation*}
		
		\item\label{estimate:equiv3}
		Whenever $\omega\in L^1_0\Lambda^l$ and
		$P\omega\in X\Lambda^l$, one has $\omega\in X\Lambda^l$ and
		\begin{equation*}
			\|\omega\|_{X\Lambda^l}
			\lesssim
			\|\omega\|_{L^1_0\Lambda^l}
			+
			\|P\omega\|_{X\Lambda^l}.
		\end{equation*}
		
		\item\label{estimate:equiv4}
		Whenever $\eta\in(X+L^1_0)\Lambda^l$ is coexact, in the sense
		that $P^\perp\eta=\eta$, one has $\eta\in X\Lambda^l$ and
		\begin{equation*}
			\|\eta\|_{X\Lambda^l}
			\lesssim
			\|\eta\|_{(X+L^1_0)\Lambda^l}.
		\end{equation*}
	\end{enumerate}
\end{prop}

\begin{proof}
	We first establish the equivalence of \ref{estimate:equiv2}, \ref{estimate:equiv3}, and \ref{estimate:equiv4}. The implication \ref{estimate:equiv2}$\Rightarrow$\ref{estimate:equiv3} follows immediately from the trivial embedding $L^1_0\subset X+L^1_0$.
	Conversely, assume \ref{estimate:equiv3}, and let
	\begin{equation*}
		\omega=\omega_0+\omega_1,
		\qquad
		\omega_0\in X\Lambda^l,
		\qquad
		\omega_1\in L^1_0\Lambda^l,
	\end{equation*}
	with $P\omega\in X\Lambda^l$. Then
	\begin{equation*}
		P\omega_1=P\omega-P\omega_0\in X\Lambda^l.
	\end{equation*}
	Applying \ref{estimate:equiv3} to $\omega_1$ and using the boundedness of $P$ on $X$, we obtain
	\begin{equation*}
		\begin{aligned}
			\|\omega\|_{X\Lambda^l}
			&\leq \|\omega_0\|_{X\Lambda^l}
			+ \|\omega_1\|_{X\Lambda^l}
			\\
			&\lesssim \|\omega_0\|_{X\Lambda^l}
			+ \|\omega_1\|_{L^1_0\Lambda^l}
			+ \|P\omega\|_{X\Lambda^l}.
		\end{aligned}
	\end{equation*}
	Taking the infimum over all such decompositions proves \ref{estimate:equiv2}.
	
	Next, if \ref{estimate:equiv2} holds and $\eta\in(X+L^1_0)\Lambda^l$ is coexact, then $P\eta=0$, and \ref{estimate:equiv4} follows directly. Conversely, assume \ref{estimate:equiv4}, and let $\omega\in L^1_0\Lambda^l$ satisfy $P\omega\in X\Lambda^l$. Then
	\begin{equation*}
		\eta=P^\perp\omega=\omega-P\omega
	\end{equation*}
	belongs to $(X+L^1_0)\Lambda^l$ and is coexact. Hence
	\begin{equation*}
		\begin{aligned}
			\|\omega\|_{X\Lambda^l}
			&\leq \|P\omega\|_{X\Lambda^l}
			+ \|P^\perp\omega\|_{X\Lambda^l}
			\\
			&\lesssim \|P\omega\|_{X\Lambda^l}
			+ \|P^\perp\omega\|_{(X+L^1_0)\Lambda^l}
			\\
			&\lesssim \|P\omega\|_{X\Lambda^l}
			+ \|\omega\|_{L^1_0\Lambda^l},
		\end{aligned}
	\end{equation*}
	which proves \ref{estimate:equiv3}.
	
	We have thus established that \ref{estimate:equiv2}$\Leftrightarrow$\ref{estimate:equiv3}$\Leftrightarrow$\ref{estimate:equiv4}. It remains only to connect these equivalent estimates to the Bourgain--Brezis selection estimate \ref{estimate:equiv1}. This is done by showing that \ref{estimate:equiv1}$\Rightarrow$\ref{estimate:equiv2} and \ref{estimate:equiv4}$\Rightarrow$\ref{estimate:equiv1}.

	We next prove that \ref{estimate:equiv1} implies \ref{estimate:equiv2}. Let $\omega\in(X+L^1_0)\Lambda^l$ satisfy $P\omega\in X\Lambda^l$, and take $v\in\mathscr S_0\Lambda^l$. By \ref{estimate:equiv1}, there exists $u\in(Y\cap L^\infty)\Lambda^l$ such that $du=dv$ and
	\begin{equation*}
		\|u\|_{(Y\cap L^\infty)\Lambda^l}
		\lesssim
		\|v\|_{Y\Lambda^l}.
	\end{equation*}
	Since $d(v-u)=0$, one has
	\begin{equation*}
		P^\perp(v-u)=0,
		\qquad
		P(v-u)=v-u.
	\end{equation*}
	Using the self-adjointness of $P$, we obtain
	\begin{equation*}
		\langle\omega,v\rangle
		=
		\langle\omega,u\rangle
		+
		\langle P\omega,v-u\rangle.
	\end{equation*}
	Therefore, by \eqref{eq:dual-sum-intersection},
	\begin{equation*}
		\begin{aligned}
			|\langle\omega,v\rangle|
			&\lesssim
			\|\omega\|_{(X+L^1_0)\Lambda^l}
			\|u\|_{(Y\cap L^\infty)\Lambda^l}
			+
			\|P\omega\|_{X\Lambda^l}
			\|v-u\|_{Y\Lambda^l}
			\\
			&\lesssim
			\left(
			\|\omega\|_{(X+L^1_0)\Lambda^l}
			+
			\|P\omega\|_{X\Lambda^l}
			\right)
			\|v\|_{Y\Lambda^l}.
		\end{aligned}
	\end{equation*}
	Since $\mathscr S_0$ is dense in $Y$ and $X=Y'$, this proves that $\omega\in X\Lambda^l$ and establishes \ref{estimate:equiv2}.
	
	It remains to show that \ref{estimate:equiv4} implies \ref{estimate:equiv1}. Let
	\begin{equation*}
		\mathcal E
		=
		\big\{
		\eta\in X\Lambda^l:P^\perp\eta=\eta
		\big\}
	\end{equation*}
	be the closed subspace of coexact forms in $X\Lambda^l$. Given
	$v\in Y\Lambda^l$, we define the linear map
	\begin{equation*}
		\mathcal L_v(\eta)=
		\langle\eta,v\rangle_{X,Y},
		\qquad
		\eta\in\mathcal E.
	\end{equation*}
	By \ref{estimate:equiv4}, we see that
	\begin{equation*}
		|\mathcal L_v(\eta)|
		\leq
		\|\eta\|_{X\Lambda^l}\|v\|_{Y\Lambda^l}
		\lesssim
		\|\eta\|_{(X+L^1_0)\Lambda^l}\|v\|_{Y\Lambda^l}.
	\end{equation*}
	The Hahn--Banach theorem therefore extends $\mathcal L_v$ to a bounded linear functional on $(X+L^1_0)\Lambda^l$.
	
	By the duality relation \eqref{eq:dual-sum-intersection}, this extension is represented by some $u\in(Y\cap L^\infty)\Lambda^l$ satisfying
	\begin{equation*}
		\|u\|_{(Y\cap L^\infty)\Lambda^l}
		\lesssim
		\|v\|_{Y\Lambda^l}.
	\end{equation*}
	Moreover, it holds that
	\begin{equation*}
		\langle\eta,v-u\rangle_{X,Y}=0
		\qquad
		\text{for every $\eta\in\mathcal E$}.
	\end{equation*}
	Now, for every $\phi\in\mathscr S_0\Lambda^l$, the form $P^\perp\phi$ belongs to $\mathcal E$. Hence, by the self-adjointness
	of $P^\perp$, we deduce that
	\begin{equation*}
		\big\langle\phi,P^\perp(v-u)\big\rangle_{X,Y}
		=
		\big\langle P^\perp\phi,v-u\big\rangle_{X,Y}
		=0.
	\end{equation*}
	It follows that $P^\perp(v-u)=0$ in $\mathscr S_0'\Lambda^l$, and therefore $du=dv$. This establishes \ref{estimate:equiv1} and thereby completes the proof.
\end{proof}

The preceding proposition shows that the remainder of the paper may be carried out on the dual side. We will most often prove \ref{estimate:equiv3} and then recover the corresponding Bourgain--Brezis selection estimate from \ref{estimate:equiv1}.

\begin{rem}
	The proposition is stated for $1<p,q<\infty$, which is the range in which the duality of the homogeneous Triebel--Lizorkin and Besov spaces is used directly. In later critical Besov statements, the case $q=1$ follows separately from the embedding
	\begin{equation*}
		\dot B^{\frac np}_{p,1}\subset L^\infty.
	\end{equation*}
\end{rem}


\section{The trace-free Beurling--Ahlfors transform}\label{section:beurling}

The purpose of this section is to isolate the operator underlying the endpoint cancellation used throughout the paper. We first explain its canonical normalization, then establish a scalar reflection-cancellation principle, and finally identify the scalar multiplier components of $S$ and deduce the fundamental estimate stated in Theorem~\ref{thm:intro:beurling}.

\subsection{Definition and canonical normalization}

Recall that
\begin{equation*}
	\widetilde S=P-P^\perp
\end{equation*}
is the generalized Beurling--Ahlfors transform on differential forms. The operator introduced in \eqref{eq:intro:S} can be written in the equivalent forms
\begin{equation}\label{eq:beurling_identities}
	S
	=
	\frac12\left(
	\widetilde S-\frac{2l-n}{n}\operatorname{Id}
	\right)
	=
	\frac{n-l}{n}P-\frac lnP^\perp
	=
	P-\frac ln\operatorname{Id}.
\end{equation}
Thus $S$ is a normalized trace-free part of $\widetilde S$. To be precise, the normalization is chosen so that the difference between its two eigenvalues is equal to one.

The following elementary proposition records the resulting algebraic properties.

\begin{prop}\label{prop:beurling_algebra}
	Let $n\geq2$ and $1\leq l\leq n-1$. For every $\xi\neq0$, the symbol $S(\xi)$ is self-adjoint and acts by
	\begin{equation*}
		\text{multiplication by }\frac{n-l}{n}
		\text{ on }\operatorname{ran}P(\xi),
	\end{equation*}
	and by
	\begin{equation*}
		\text{multiplication by }-\frac ln
		\text{ on }\operatorname{ran}P^\perp(\xi).
	\end{equation*}
	In particular,
	\begin{equation*}
		\operatorname{tr}S(\xi)=0.
	\end{equation*}
	Moreover, $S$ is the unique operator of the form
	\begin{equation*}
		T=aP+bP^\perp
	\end{equation*}
	whose symbol is trace-free and whose two eigenvalues satisfy $a-b=1$.
	The operator $S$ is invertible, with
	\begin{equation*}
		S^{-1}
		=
		\frac n{n-l}P-\frac nlP^\perp.
	\end{equation*}
\end{prop}

\begin{proof}
	The first assertions follow immediately from the Hodge splitting.
	
	Next, by \eqref{eq:hodge_ranks}, we see that
	\begin{equation*}
		\operatorname{rank}P(\xi)
		=
		\frac ln\dim\Lambda^l,
		\qquad
		\operatorname{rank}P^\perp(\xi)
		=
		\frac{n-l}{n}\dim\Lambda^l.
	\end{equation*}
	Hence, if $T=aP+bP^\perp$ has trace-free symbol and $a-b=1$, then
	\begin{equation*}
		la+(n-l)b=0,
		\qquad
		a-b=1.
	\end{equation*}
	Solving this system gives
	\begin{equation*}
		a=\frac{n-l}{n},
		\qquad
		b=-\frac ln,
	\end{equation*}
	which establishes the canonical choice of coefficients in the definition of $S$.
	
	Finally, the formula for $S^{-1}$ follows by inverting these two nonzero eigenvalues, which completes the proof.
\end{proof}

Since $S$ is a linear combination of the Hodge projections, it is bounded on all the Triebel--Lizorkin and Besov spaces in the ranges described in Section~\ref{subsec:hodge_projections}. Notice also that
\begin{equation}\label{eq:recover_identity_from_S}
	\operatorname{Id}
	=
	\frac nl(P-S).
\end{equation}
This elementary identity will be used repeatedly below.

The relation with the generalized Beurling--Ahlfors transform is especially
simple in the middle degree. If $n=2l$, then
\begin{equation*}
	S=\frac12\widetilde S.
\end{equation*}
In particular, for $n=2$ and $l=1$, the transform $S$ agrees, up to a normalizing factor and the standard identification of $1$-forms with complex functions, with the classical planar Beurling--Ahlfors transform.

\subsection{Reflection cancellation}

We now isolate the scalar multiplier estimate that drives the analysis. For $\sigma\in\mathbb S^{n-1}$, let
\begin{equation*}
	\rho_\sigma\xi
	=
	\xi-2(\xi\cdot\sigma)\sigma
\end{equation*}
be the orthogonal reflection across the hyperplane $\sigma^\perp$.

\begin{lem}[Reflection cancellation]\label{lemma:bound:oddness}
	Let $m\in C^\infty(\mathbb R^n\setminus\{0\})$, with $n\geq 1$, be homogeneous of degree
	zero. Assume that, for some $\sigma\in\mathbb S^{n-1}$,
	\begin{equation}\label{eq:reflection_oddness}
		m(\rho_\sigma\xi)=-m(\xi)
		\qquad
		\text{for every $\xi\neq0$}.
	\end{equation}
	Then
	\begin{equation*}
		T_m=|D|^{-n}m(D):\mathscr M(\mathbb R^n)\rightarrow
		L^\infty(\mathbb R^n)
	\end{equation*}
	is bounded. Moreover,
	\begin{equation*}
		T_m[L^1]\subset C,
		\qquad
		T_m[L^1_0]\subset C_0.
	\end{equation*}
\end{lem}

Before proceeding to the proof, observe that $T_m$ is well defined on homogeneous distributions. Indeed, if $\phi\in\mathscr S_0$, then $\widehat\phi$ and all its derivatives vanish to infinite order at the origin. Since $|\xi|^{-n}m(\xi)$ is smooth away from the origin and homogeneous of degree $-n$, the product
\begin{equation*}
	|\xi|^{-n}m(\xi)\widehat\phi(\xi)
\end{equation*}
extends to a Schwartz function which vanishes to infinite order at the origin. Hence
\begin{equation*}
	T_m[\mathscr S_0]\subset\mathscr S_0,
\end{equation*}
and $T_m$ acts on $\mathscr S_0'$ by duality.

\begin{proof}
	The argument is a dyadic refinement of the reflection cancellation used in \cite[Section~4]{bourgain:brezis:2003}.
	
	Let $\varphi$ be the radial Littlewood--Paley cutoff fixed in Section~\ref{sec:preliminaries} and set
	\begin{equation*}
		\mathcal K
		=
		\mathscr F^{-1}\left(
		\varphi(\xi)|\xi|^{-n}m(\xi)
		\right).
	\end{equation*}
	Since the multiplier in parentheses is smooth and compactly supported away from the origin, it holds that $\mathcal K\in\mathscr S_0$. Homogeneity then gives
	\begin{equation}\label{eq:dyadic_kernel_representation}
		\Delta_jT_m\mu(x)
		=
		\int_{\mathbb R^n}\mathcal K\big(2^j(x-y)\big)d\mu(y).
	\end{equation}

	Because $\varphi$ and $|\xi|^{-n}$ are radial, the reflection property \eqref{eq:reflection_oddness} implies
	\begin{equation*}
		\mathcal K(\rho_\sigma x)=-\mathcal K(x).
	\end{equation*}
	In particular, $\mathcal K$ vanishes on $\sigma^\perp$. By the mean value theorem in the direction $\sigma$, combined with the rapid decay of $\mathcal K$, one obtains, for every $N>1$,
	\begin{equation*}
		|\mathcal K(x)|
		\lesssim \frac{|x\cdot\sigma|}{(1+|x|)^N}\leq
		\frac{|x|}{(1+|x|)^N}.
	\end{equation*}
	Consequently,
	\begin{equation}\label{bound:basic}
		\sup_{x\in\mathbb R^n}
		\sum_{j\in\mathbb Z}
		|\mathcal K(2^jx)|
		\lesssim
		\sup_{x\in\mathbb R^n}
		\sum_{j\in\mathbb Z}
		\frac{|2^jx|}{1+|2^jx|^N}
		\lesssim 1.
	\end{equation}
	Indeed, for $x\neq0$, one splits the sum according to whether $2^j|x|\leq 1$ or $2^j|x|>1$ to deduce a uniform bound.
	
	It now follows from \eqref{eq:dyadic_kernel_representation} and
	\eqref{bound:basic} that
	\begin{equation}\label{estimate:stronger_uniform_convergence}
		\sum_{j\in\mathbb Z}|\Delta_jT_m\mu(x)|
		\leq
		\int_{\mathbb R^n}
		\sum_{j\in\mathbb Z}
		\big|\mathcal K\big(2^j(x-y)\big)\big|d|\mu|(y)
		\lesssim
		\|\mu\|_{\mathscr M}.
	\end{equation}
	The dyadic series therefore converges absolutely pointwise and represents $T_m\mu$ in $\mathscr S_0'$. This proves
	\begin{equation*}
		\|T_m\mu\|_{L^\infty}
		\lesssim
		\|\mu\|_{\mathscr M}.
	\end{equation*}

	If $f\in L^1$, translation invariance and the preceding estimate give
	\begin{equation*}
		\|\tau_hT_mf-T_mf\|_{L^\infty}
		\lesssim
		\|\tau_hf-f\|_{L^1},
	\end{equation*}
	where $\tau_hf(x)=f(x-h)$, which tends to zero as $h\to0$. Thus $T_mf$ has a bounded uniformly continuous representative.
	
	Finally, if $f\in L^1_0$, choose an approximating sequence $f_k\in\mathscr S_0$ with $f_k\to f$ in $L^1$. Since $T_m[\mathscr S_0]\subset\mathscr S_0$, the functions $T_mf_k$ belong to $C_0$. Then, the boundedness of $T_m:L^1\to L^\infty$ shows that $T_mf_k\to T_mf$ uniformly, and hence $T_mf\in C_0$, which concludes the proof.
\end{proof}

The previous proof gives a strengthened Littlewood--Paley estimate that will be used in the Besov and multilinear arguments.

\begin{cor}\label{scholium:bound_m}
	Under the assumptions of Lemma~\ref{lemma:bound:oddness},
	\begin{equation*}
		\Big\|
		\sum_{j\in\mathbb Z}
		\big|\Delta_j|D|^{-n}m(D)\mu\big|
		\Big\|_{L^\infty}
		\lesssim
		\|\mu\|_{\mathscr M}.
	\end{equation*}
	If $\mu=f$ belongs to $L^1$ or $L^1_0$, then the function
	\begin{equation*}
		\sum_{j\in\mathbb Z}
		\big|\Delta_j|D|^{-n}m(D)f\big|
	\end{equation*}
	belongs to $C$ or $C_0$, respectively. The same conclusions hold with $\Delta_j^2$ in place of $\Delta_j$.
\end{cor}

\begin{proof}
	The $L^\infty$ estimate is \eqref{estimate:stronger_uniform_convergence}. Equivalently, the map
	\begin{equation*}
		f\mapsto
		\big(\Delta_jT_mf\big)_{j\in\mathbb Z}
	\end{equation*}
	is bounded from $L^1$ into $L^\infty(\ell^1)$.
	
	As before, translation continuity in $L^1$ therefore gives continuity of this $\ell^1$-valued function. Moreover, if $f\in L^1_0$, we can approximate it in $L^1$ by elements of $\mathscr S_0$. Then, for such an element $g\in \mathscr S_0$, one has $T_mg\in\mathscr S_0$, and the rapid decay of its dyadic blocks at both high and low frequencies shows that $\big(\Delta_jT_mg\big)_j$ belongs to $C_0(\ell^1)$. The corresponding vanishing at infinity assertion for $f$ then follows by uniform convergence.
	
	Finally, we note that the same arguments apply with $\Delta_j^2$ in place of $\Delta_j$, which concludes the proof.
\end{proof}

\subsection{Multiplier structure and endpoint bounds}

We next show that the trace-free normalization of the Beurling--Ahlfors transform forces every scalar component of $S$ to have the reflection symmetry required in Lemma~\ref{lemma:bound:oddness}. Let
\begin{equation*}
	\varepsilon_i\omega=dx_i\wedge\omega,
	\qquad
	\iota_i\omega=\iota_{e_i}\omega,
	\qquad
	A_i=\varepsilon_i\iota_i.
\end{equation*}
On $\Lambda^l$, one has
\begin{equation}\label{eq:sum_epsilon_iota}
	\sum_{i=1}^nA_i=l\operatorname{Id}.
\end{equation}
Indeed, $A_i$ acts as the identity on a basis form $dx_I$ when $i\in I$ and
annihilates it when $i\notin I$.

\begin{prop}[Multiplier decomposition]\label{prop:beurling_multiplier_decomposition}
	For every $\xi\neq0$,
	\begin{equation}\label{eq:beurling_multiplier_decomposition}
		S(\xi)=
		\frac1n
		\sum_{1\leq i<j\leq n}
		\frac{\xi_i^2-\xi_j^2}{|\xi|^2}
		(A_i-A_j)
		+\sum_{\substack{1\leq i,j\leq n\\ i\neq j}}
		\frac{\xi_i\xi_j}{|\xi|^2}
		\varepsilon_i\iota_j.
	\end{equation}
	
	Consequently, every matrix coefficient of $S(\xi)$ is a finite linear
	combination of scalar multipliers of the forms
	\begin{equation*}
		m_{ij}^{\sharp}(\xi)
		=
		\frac{\xi_i^2-\xi_j^2}{|\xi|^2},
		\qquad
		m_{ij}^{\flat}(\xi)
		=
		\frac{\xi_i\xi_j}{|\xi|^2},
		\qquad
		i\neq j.
	\end{equation*}
	Each of these multipliers is odd under an orthogonal reflection.
\end{prop}

\begin{proof}
	By \eqref{eq:beurling_identities} and the symbol formula for $P$,
	\begin{equation*}
		S(\xi)
		=
		\frac1{|\xi|^2}
		\sum_{i,j=1}^n\xi_i\xi_j\varepsilon_i\iota_j
		-\frac ln\operatorname{Id}.
	\end{equation*}
	Using \eqref{eq:sum_epsilon_iota}, the diagonal part of the sum becomes
	\begin{equation*}
		\sum_{i=1}^n
		\left(
		\frac{\xi_i^2}{|\xi|^2}-\frac1n
		\right)A_i
		=
		\frac1n
		\sum_{1\leq i<j\leq n}
		\frac{\xi_i^2-\xi_j^2}{|\xi|^2}(A_i-A_j),
	\end{equation*}
	which proves \eqref{eq:beurling_multiplier_decomposition}.
	
	The multiplier $m_{ij}^{\flat}$ is odd under reflection across $e_i^\perp$. The multiplier $m_{ij}^{\sharp}$ is odd under the reflection that exchanges the $i$th and $j$th coordinates, namely the reflection across the hyperplane orthogonal to $(e_i-e_j)/\sqrt2$. This concludes the proof.
\end{proof}

We can now prove the endpoint estimate announced in the introduction.

\begin{prop}[Endpoint cancellation for $S$]\label{prop:Beurling_transform}
	Let $n\geq2$ and $1\leq l\leq n-1$. Then
	\begin{equation*}
		|D|^{-n}S:
		\mathscr M \Lambda^l
		\rightarrow
		L^\infty \Lambda^l
	\end{equation*}
	is bounded. Moreover,
	\begin{equation*}
		|D|^{-n}S[L^1\Lambda^l]\subset C\Lambda^l,
		\qquad
		|D|^{-n}S[L^1_0\Lambda^l]\subset C_0\Lambda^l.
	\end{equation*}
\end{prop}

\begin{proof}
	By Proposition~\ref{prop:beurling_multiplier_decomposition}, each coefficient of $S\omega$ is a finite linear combination of scalar multipliers satisfying the hypotheses of Lemma~\ref{lemma:bound:oddness}. Applying that lemma coefficientwise proves the result.
\end{proof}

Proposition~\ref{prop:Beurling_transform} proves Theorem~\ref{thm:intro:beurling}. Its dyadic strengthening is the form that will enter the Besov estimates.

\begin{cor}\label{scholium:bound_S}
	Under the assumptions of Proposition~\ref{prop:Beurling_transform},
	\begin{equation*}
		\Big\|
		\sum_{j\in\mathbb Z}
		\big|\Delta_j|D|^{-n}S\mu\big|
		\Big\|_{L^\infty}
		\lesssim
		\|\mu\|_{\mathscr M\Lambda^l}.
	\end{equation*}
	If $\mu=\omega$ belongs to $L^1\Lambda^l$ or $L^1_0\Lambda^l$, then the function
	\begin{equation*}
		\sum_{j\in\mathbb Z}
		\big|\Delta_j|D|^{-n}S\omega\big|
	\end{equation*}
	belongs to $C$ or $C_0$, respectively. The same conclusions hold with
	$\Delta_j^2$ in place of $\Delta_j$.
\end{cor}

\begin{proof}
	This follows coefficientwise from Corollary~\ref{scholium:bound_m} and Proposition~\ref{prop:beurling_multiplier_decomposition}.
\end{proof}


\section{Hilbertian estimates and a Besov refinement}
\label{section:hilbertian}

The endpoint cancellation established in the preceding section allows one to transfer control of the exact Hodge component to the coexact component. We first carry this out in the Hilbertian setting. The dyadic strengthening in Corollary~\ref{scholium:bound_S} then yields a refinement in the Besov scale.

We will use the following regularization in all arguments. Fix a radial function $\chi\in C_c^\infty(\mathbb R^n)$ which is equal to one in a neighborhood of the origin, and define
\begin{equation*}
	Q_N
	=\chi(2^{-N}D)-\chi(2^ND),
	\qquad
	N\geq1.
\end{equation*}
The operators $Q_N$ commute with $P$, $P^\perp$, $S$, and the dyadic blocks. They are uniformly bounded on $L^1$ and all the spaces considered below. Moreover, the symbol of $Q_N$ is supported in a compact annulus. Thus all computations may first be performed rigorously with $Q_N\omega$. The resulting estimates are uniform in $N$, and the general statements follow by letting $N\to\infty$. We will therefore carry out the proofs for forms with Fourier support in a compact annulus and omit further reference to this regularization.

\subsection{The Hilbertian estimate}

The first result is the dual estimate underlying the full Hilbertian case of the Bourgain--Brezis conjecture.

\begin{thm}[Hilbertian dual estimate]\label{thm:hilbertian_estimate}
	Let $n\geq2$ and $1\leq l\leq n-1$. Whenever $\omega\in L^1_0\Lambda^l$ and $P\omega\in\dot H^{-\frac n2}\Lambda^l$, one has $\omega\in\dot H^{-\frac n2}\Lambda^l$ and
	\begin{equation*}
		\|\omega\|_{\dot H^{-\frac n2}\Lambda^l}
		\lesssim
		\|\omega\|_{L^1_0\Lambda^l}
		+
		\|P\omega\|_{\dot H^{-\frac n2}\Lambda^l}.
	\end{equation*}
\end{thm}

\begin{proof}
	Since $P$, $P^\perp$, $S$, and $|D|$ commute, and since the Hodge projections are self-adjoint, one has
	\begin{equation}\label{eq:hilbertian_beurling_identity}
		\frac{n-l}{n}\|P\omega\|_{\dot H^{-\frac n2}\Lambda^l}^2
		-\frac l n\|P^\perp\omega\|_{\dot H^{-\frac n2}\Lambda^l}^2
		=
		\int_{\mathbb R^n}
		\left\langle
		\omega,|D|^{-n}S\omega
		\right\rangle dx.
	\end{equation}
	Proposition~\ref{prop:Beurling_transform} gives
	\begin{equation*}
		\left|
		\int_{\mathbb R^n}
		\left\langle
		\omega,|D|^{-n}S\omega
		\right\rangle dx
		\right|
		\leq
		\|\omega\|_{L^1_0\Lambda^l}
		\big\||D|^{-n}S\omega\big\|_{L^\infty\Lambda^l}
		\lesssim
		\|\omega\|_{L^1_0\Lambda^l}^2.
	\end{equation*}
	It then follows from \eqref{eq:hilbertian_beurling_identity} that
	\begin{equation*}
		\|P^\perp\omega\|_{\dot H^{-\frac n2}\Lambda^l}^2
		\lesssim
		\|\omega\|_{L^1_0\Lambda^l}^2
		+
		\|P\omega\|_{\dot H^{-\frac n2}\Lambda^l}^2.
	\end{equation*}
	Finally, the exact and coexact components are orthogonal in
	$\dot H^{-\frac n2}\Lambda^l$, and hence
	\begin{equation*}
		\|\omega\|_{\dot H^{-\frac n2}\Lambda^l}^2
		=
		\|P\omega\|_{\dot H^{-\frac n2}\Lambda^l}^2
		+
		\|P^\perp\omega\|_{\dot H^{-\frac n2}\Lambda^l}^2.
	\end{equation*}
	The desired estimate follows.
\end{proof}

The duality principle now gives the Hilbertian Bourgain--Brezis theorem in all dimensions and form degrees.

\begin{cor}[Hilbertian Bourgain--Brezis estimate]\label{cor:hilbertian_solution}
	Let $n\geq2$ and $1\leq l\leq n-1$. For every $v\in\dot H^{\frac n2}\Lambda^l$, there exists $u\in(\dot H^{\frac n2}\cap L^\infty)\Lambda^l$ such that $du=dv$ and
	\begin{equation*}
		\|u\|_{(\dot H^{\frac n2}\cap L^\infty)\Lambda^l}
		\lesssim
		\|v\|_{\dot H^{\frac n2}\Lambda^l}.
	\end{equation*}
	In particular,
	\begin{equation*}
		d\big[\dot H^{\frac n2}\Lambda^l\big]
		=
		d\big[(\dot H^{\frac n2}\cap L^\infty)\Lambda^l\big]
	\end{equation*}
	with equivalent norms.
\end{cor}

\begin{proof}
	Theorem~\ref{thm:hilbertian_estimate} is assertion \ref{estimate:equiv3} of Proposition~\ref{prop:duality} with
	\begin{equation*}
		X=\dot H^{-\frac n2},
		\qquad
		Y=\dot H^{\frac n2}.
	\end{equation*}
	The conclusion is assertion \ref{estimate:equiv1} of that proposition.
\end{proof}

\subsection{A dyadic Besov refinement}

The Hilbertian identity can be localized to each dyadic annulus. The $L^\infty(\ell^1)$ estimate from Corollary~\ref{scholium:bound_S} then
allows the localized bounds to be summed over all dyadic blocks.

\begin{thm}[Besov dual estimate]\label{thm:hilbertian_estimate_Besov}
	Let $n\geq2$, $1\leq l\leq n-1$, and $2\leq r\leq\infty$.
	Whenever $\omega\in L^1_0\Lambda^l$ and
	$P\omega\in\dot B^{-\frac n2}_{2,r}\Lambda^l$, one has
	$\omega\in\dot B^{-\frac n2}_{2,r}\Lambda^l$ and
	\begin{equation*}
		\|\omega\|_{\dot B^{-\frac n2}_{2,r}\Lambda^l}
		\lesssim
		\|\omega\|_{L^1_0\Lambda^l}
		+
		\|P\omega\|_{\dot B^{-\frac n2}_{2,r}\Lambda^l}.
	\end{equation*}
\end{thm}

\begin{proof}
	Note first that the case $r=2$ corresponds to Theorem~\ref{thm:hilbertian_estimate}, while the endpoint case $r=\infty$ follows from the embedding
	\begin{equation*}
		L^1\subset \dot B^{-\frac n2}_{2,\infty}.
	\end{equation*}

	If $2<r<\infty$, set
	\begin{equation*}
		a_j
		=
		\|\Delta_jP^\perp\omega\|_{\dot H^{-\frac n2}\Lambda^l},
		\qquad
		b_j
		=
		\|\Delta_jP\omega\|_{\dot H^{-\frac n2}\Lambda^l},
	\end{equation*}
	and, by frequency localization, observe that
	\begin{equation*}
		\|P^\perp\omega\|_{\dot B^{-\frac n2}_{2,r}\Lambda^l}
		\sim
		\|(a_j)_{j\in\mathbb Z}\|_{\ell^r},
		\qquad
		\|P\omega\|_{\dot B^{-\frac n2}_{2,r}\Lambda^l}
		\sim
		\|(b_j)_{j\in\mathbb Z}\|_{\ell^r}.
	\end{equation*}
	We then apply \eqref{eq:hilbertian_beurling_identity} to $\Delta_j\omega$ to deduce
	\begin{equation*}
		\begin{aligned}
			\sum_{j\in\mathbb{Z}}\left|\frac{n-l}nb_j^2
			-\frac lna_j^2\right|^\frac r2
			=\sum_{j\in\mathbb{Z}}\left|
			\int_{\mathbb{R}^n}\langle \omega,\Delta_j^2|D|^{-n}S\omega \rangle dx\right|^\frac r2.
		\end{aligned}
	\end{equation*}
	Since $r/2> 1$, we obtain
	\begin{equation*}
		\|(a_j)_{j\in\mathbb Z}\|_{\ell^r}^2\lesssim \|(b_j)_{j\in\mathbb Z}\|_{\ell^r}^2
		+\sum_{j\in\mathbb{Z}}\left|
		\int_{\mathbb{R}^n}\langle \omega,\Delta_j^2|D|^{-n}S\omega \rangle dx\right|.
	\end{equation*}
	It then follows that
	\begin{equation*}
		\|(a_j)_{j\in\mathbb Z}\|_{\ell^r}^2\lesssim \|(b_j)_{j\in\mathbb Z}\|_{\ell^r}^2
		+\|\omega\|_{L^1_0\Lambda^l}
		\Big\|\sum_{j\in\mathbb{Z}}\big|\Delta_j^2 |D|^{-n}S\omega\big|\Big\|_{C_0},
	\end{equation*}
	which, by Corollary \ref{scholium:bound_S}, gives the desired estimate.
\end{proof}

By duality, the preceding theorem gives the following critical Besov Bourgain--Brezis estimates.

\begin{cor}[Hilbertian Besov Bourgain--Brezis estimates]\label{cor:hilbertian_Besov_solution}
	Let $n\geq2$, $1\leq l\leq n-1$, and $1\leq q\leq2$. For every
	$v\in\dot B^{\frac n2}_{2,q}\Lambda^l$, there exists
	$u\in(\dot B^{\frac n2}_{2,q}\cap L^\infty)\Lambda^l$ such that
	$du=dv$ and
	\begin{equation*}
		\|u\|_{(\dot B^{\frac n2}_{2,q}\cap L^\infty)\Lambda^l}
		\lesssim
		\|v\|_{\dot B^{\frac n2}_{2,q}\Lambda^l}.
	\end{equation*}
	In particular,
	\begin{equation*}
		d\big[\dot B^{\frac n2}_{2,q}\Lambda^l\big]
		=
		d\big[(\dot B^{\frac n2}_{2,q}\cap L^\infty)\Lambda^l\big]
	\end{equation*}
	with equivalent norms.
\end{cor}

\begin{proof}
	If $1<q\leq2$, set $r=q'$. Theorem~\ref{thm:hilbertian_estimate_Besov} is assertion \ref{estimate:equiv3} of Proposition~\ref{prop:duality} with
	\begin{equation*}
		X=\dot B^{-\frac n2}_{2,r},
		\qquad
		Y=\dot B^{\frac n2}_{2,q}.
	\end{equation*}
	Assertion \ref{estimate:equiv1} gives the conclusion.
	
	If $q=1$, the result follows directly from
	\begin{equation*}
		\dot B^{\frac n2}_{2,1}\subset L^\infty,
	\end{equation*}
	by taking $u=v$.
\end{proof}

The preceding arguments are quadratic. They rely on the Hilbertian identity \eqref{eq:hilbertian_beurling_identity} and its dyadic localization. To move away from the exponent $2$, we replace these identities by multilinear expansions. This is the purpose of the next section.


\section{Multilinear reflection estimates}\label{section:multilinear}

The estimates of the preceding section are quadratic. We now replace the Hilbertian pairing by multilinear expansions. The reflection symmetry of the scalar components of $S$ survives in these expansions and gives the non-Hilbertian estimates stated in the introduction.

\subsection{A multiscale reflection estimate}

We begin with the scalar estimate used in both the Besov and Triebel--Lizorkin arguments.

\begin{lem}[Multilinear reflection cancellation]\label{lemma:multilinear_reflection}
	Let $m\in C^\infty(\mathbb R^n\setminus\{0\})$, with $n\geq 1$, be homogeneous of degree zero and suppose that the reflection property \eqref{eq:reflection_oddness} holds, for some $\sigma\in\mathbb S^{n-1}$.
	Let $M\geq2$ and $K\geq1$ be integers, and set
	\begin{equation*}
		L=MK-1.
	\end{equation*}
	Then, for all $f,g\in L^1(\mathbb R^n)$,
	\begin{equation}\label{estimate:multilinear}
		\sum_{j_1\geq\cdots\geq j_K}2^{-\frac{nL}{K}(j_1+\cdots+j_K)}
		\bigg|
		\int_{\mathbb R^n}
		\big(\Delta_{j_1}m(D)f\big)
		(\Delta_{j_1}g)^{M-1}
		\prod_{k=2}^K(\Delta_{j_k}g)^M dx
		\bigg|
		\lesssim
		\|f\|_{L^1}\|g\|_{L^1}^{L}.
	\end{equation}
	When $K=1$, the product over $k=2,\ldots,K$ is understood to be empty.
\end{lem}

\begin{proof}
	Let
	\begin{equation*}
		G_1=\{1,\ldots,M-1\},
	\end{equation*}
	and, for $2\leq k\leq K$, let
	\begin{equation*}
		G_k=\{M(k-1),\ldots,Mk-1\}.
	\end{equation*}
	Thus the sets $G_k$ partition $\{1,\ldots,L\}$. We write $\kappa(s)=k$ when $s\in G_k$.
	
	We first argue for Schwartz functions. A direct Fourier calculation gives
	\begin{equation*}
		\begin{aligned}
			&2^{-\frac{nL}{K}(j_1+\cdots+j_K)}
			\int_{\mathbb R^n}
			\big(\Delta_{j_1}m(D)f\big)
			(\Delta_{j_1}g)^{M-1}
			\prod_{k=2}^K(\Delta_{j_k}g)^M dx
			\\
			&\qquad=
			\int_{\mathbb R^n}f(x)
			\left(\int_{(\mathbb R^n)^L}
			\mathcal K(x-y_1,\ldots,x-y_L)
			\prod_{s=1}^L g(y_s)
			dy_1\cdots dy_L\right)dx.
		\end{aligned}
	\end{equation*}
	Here, we rewrite the kernel $\mathcal K$ as
	\begin{equation}\label{eq:multiscale_kernel_form}
		\mathcal K(z_1,\ldots,z_L)=
		\left(
		\prod_{k=2}^K2^{-\frac nK d_k}
		\right)
		\widetilde{\mathcal K}
		\big(
		2^{j_{\kappa(1)}}z_1,\ldots,
		2^{j_{\kappa(L)}}z_L
		\big),
	\end{equation}
	with
	\begin{equation*}
		d_1=0,
		\qquad
		d_k=j_1-j_k
		\quad(2\leq k\leq K).
	\end{equation*}
	Since $j_1\geq\cdots\geq j_K$, one has $0=d_1\leq d_2\leq\cdots\leq d_K$.
	
	After rescaling the Fourier variables indexed in $G_k$ by $2^{j_k}$ and using the homogeneity of $m$, a computation shows that the kernel $\widetilde{\mathcal K}$ is the inverse Fourier transform on $(\mathbb R^n)^L$ of the symbol
	\begin{equation*}
		\left(\prod_{s=1}^L\varphi(\eta_s)\right)
		\varphi(\chi)m(-\chi),
	\end{equation*}
	where $\eta_s\in\mathbb R^n$ is the Fourier variable dual to $z_s$ and
	\begin{equation*}
		\chi=\sum_{s=1}^L2^{-d_{\kappa(s)}}\eta_s.
	\end{equation*}
	Notice that the family of these symbols is bounded in $C_c^\infty\big((\mathbb R^n)^L\big)$, uniformly in $0=d_1\leq\cdots\leq d_K$. The cutoff $\varphi(\chi)$ keeps the argument of $m$ away from the origin. Consequently, the family $\widetilde{\mathcal K}$ is uniformly bounded in $\mathscr S\big((\mathbb R^n)^L\big)$.
	
	The reflection hypothesis also gives
	\begin{equation*}
		\widetilde{\mathcal K}(\rho_\sigma z_1,\ldots,\rho_\sigma z_L)
		=-\widetilde{\mathcal K}(z_1,\ldots,z_L).
	\end{equation*}
	Therefore, arguing as in Lemma~\ref{lemma:bound:oddness}, we obtain, for every sufficiently large $Q$,
	\begin{equation}\label{eq:multiscale_kernel_odd_bound}
		\big|\widetilde{\mathcal K}(z_1,\ldots,z_L)\big|
		\lesssim
		\frac{|(z_1,\ldots,z_L)|}
		{(1+|(z_1,\ldots,z_L)|)^Q},
	\end{equation}
	uniformly in $j_1\geq\cdots\geq j_K$. If
	\begin{equation*}
		h_Q(t)=\frac{t}{(1+t)^Q},
	\end{equation*}
	then the right-hand side of \eqref{eq:multiscale_kernel_odd_bound} is bounded by
	\begin{equation*}
		\sum_{s=1}^Lh_Q(|z_s|).
	\end{equation*}
	Combining this with \eqref{eq:multiscale_kernel_form} and then using
	\begin{equation*}
		\sup_{z\in\mathbb R^n} \sum_{j\in\mathbb Z}h_Q(2^j|z|)
		\lesssim1,
	\end{equation*}
	we find that
	\begin{equation}\label{eq:multiscale_kernel_sum}
		\begin{aligned}
			\sum_{j_1\geq\cdots\geq j_K}
			\big|\mathcal K(z_1,\ldots,z_L)\big|
			&\lesssim
			\sum_{s=1}^L\sum_{j_1\geq\cdots\geq j_K}
			\left(
			\prod_{k=2}^K2^{-\frac nK d_k}
			\right)
			h_Q(2^{j_{\kappa(s)}}|z_s|)
			\\
			&\lesssim
			\sum_{s=1}^L
			\left(
			\prod_{k=2}^K\sum_{d_k\geq 0}2^{-\frac nK d_k}
			\right)
			\sum_{j\in\mathbb{Z}}h_Q(2^{j}|z_s|)
			\lesssim1,
		\end{aligned}
	\end{equation}
	uniformly over all $z_1,\ldots,z_L\in\mathbb R^n$.
	
	The estimate \eqref{estimate:multilinear} now follows from \eqref{eq:multiscale_kernel_sum} and Fubini's theorem. Finally, the general case follows by approximation in $L^1$.
\end{proof}

The case $K=1$ of Lemma~\ref{lemma:multilinear_reflection} is the diagonal multilinear estimate
\begin{equation*}
	\sum_{j\in\mathbb Z}2^{-jn(M-1)}
	\bigg|
	\int_{\mathbb R^n}
	\big(\Delta_jm(D)f\big)(\Delta_jg)^{M-1}dx
	\bigg|
	\lesssim
	\|f\|_{L^1}\|g\|_{L^1}^{M-1}.
\end{equation*}
This is the form used in the Besov argument. The full multiscale estimate is needed in the Triebel--Lizorkin setting.

\subsection{Besov estimates}

We first apply the diagonal estimate. The evenness assumption below permits us to express the $L^r$ norm of each dyadic block by a multilinear pairing.

Here, for convenience, we introduce the regularity parameter
\begin{equation}\label{notation:regularity}
	s_r=-n\Big(1-\frac1r\Big),
\end{equation}
for any $1\leq r<\infty$, and $s_\infty=-n$.

\begin{thm}[Multilinear Besov estimate]\label{thm:estimate_Besov_multilinear}
	Let $n\geq2$, $1\leq l\leq n-1$, and
	\begin{equation*}
		2\leq r\leq t\leq\infty.
	\end{equation*}
	Assume that $r$ is an even integer. Whenever $\omega\in L^1_0\Lambda^l$ and $P\omega\in\dot B^{s_r}_{r,t}\Lambda^l$, one has $\omega\in\dot B^{s_r}_{r,t}\Lambda^l$ and
	\begin{equation*}
		\|\omega\|_{\dot B^{s_r}_{r,t}\Lambda^l}
		\lesssim
		\|\omega\|_{L^1_0\Lambda^l}
		+
		\|P\omega\|_{\dot B^{s_r}_{r,t}\Lambda^l}.
	\end{equation*}
\end{thm}

\begin{proof}
	In view of the embedding $L^1\subset\dot B^{s_r}_{r,\infty}$, the case $t=\infty$ is trivial. We may thus assume that $t<\infty$.
	Moreover, using the annular regularization from Section~\ref{section:hilbertian},
	it is enough to justify the estimate for a form with compact Fourier
	support away from the origin.
	
	Fix a component $I$ and set
	\begin{equation*}
		u=\omega_I,
		\qquad
		v=(P\omega)_I.
	\end{equation*}
	From \eqref{eq:recover_identity_from_S},
	\begin{equation*}
		\Delta_ju
		=
		\frac nl\big(
		\Delta_jv-\Delta_j(S\omega)_I
		\big).
	\end{equation*}
	Since $r$ is even, multiplication by $(\Delta_ju)^{r-1}$ and
	integration give
	\begin{equation*}
		\|\Delta_ju\|_{L^r}^r
		\lesssim
		\left|
		\int_{\mathbb R^n}
		\Delta_j(S\omega)_I(\Delta_ju)^{r-1}dx
		\right|
		+
		\|\Delta_jv\|_{L^r}\|\Delta_ju\|_{L^r}^{r-1}.
	\end{equation*}
	
	After multiplying by $2^{-jn(r-1)}$, taking the $\ell^\frac tr$ norm, and using H\"older's inequality in the sequence variable, we obtain
	\begin{equation*}
		\|u\|_{\dot B_{r,t}^{s_r}}^r
		\lesssim
		\sum_{j\in\mathbb{Z}}2^{-jn(r-1)}
		\left|\int_{\mathbb{R}^n}\Delta_j(S\omega)_I(\Delta_ju)^{r-1}dx\right|
		+\|v\|_{\dot B_{r,t}^{s_r}}\|u\|_{\dot B_{r,t}^{s_r}}^{r-1},
	\end{equation*}
	where we used that $\ell^1\subset\ell^\frac tr$, for $t\geq r$.
	
	Now, recall from Proposition~\ref{prop:beurling_multiplier_decomposition} that every coefficient of $S\omega$ is a finite linear combination of $m(D)\omega_J$, where $m$ satisfies the hypotheses of Lemma~\ref{lemma:multilinear_reflection}. Applying that lemma with $M=r$ and $K=1$ yields
	\begin{equation*}
		\|u\|_{\dot B^{s_r}_{r,t}}^r
		\lesssim
		\|\omega\|_{L^1_0\Lambda^l}^r
		+
		\|v\|_{\dot B^{s_r}_{r,t}}
		\|u\|_{\dot B^{s_r}_{r,t}}^{r-1}.
	\end{equation*}
	Finally, Young's inequality and summation over the finitely many components of $\omega$ prove the result.
\end{proof}

The duality principle gives the corresponding critical result.

\begin{cor}[Besov Bourgain--Brezis estimates]\label{cor:Besov_solution}
	Let $n\geq2$, $1\leq l\leq n-1$, $1<p\leq2$, and
	$1\leq q\leq p$. Assume that $p'$ is an even integer. Then, for every
	$v\in\dot B^{\frac np}_{p,q}\Lambda^l$, there exists
	$u\in(\dot B^{\frac np}_{p,q}\cap L^\infty)\Lambda^l$ such that
	$du=dv$ and
	\begin{equation*}
		\|u\|_{(\dot B^{\frac np}_{p,q}\cap L^\infty)\Lambda^l}
		\lesssim
		\|v\|_{\dot B^{\frac np}_{p,q}\Lambda^l}.
	\end{equation*}
	In particular,
	\begin{equation*}
		d\big[\dot B^{\frac np}_{p,q}\Lambda^l\big]
		=
		d\big[(\dot B^{\frac np}_{p,q}\cap L^\infty)\Lambda^l\big]
	\end{equation*}
	with equivalent norms.
\end{cor}

\begin{proof}
	If $1<q\leq p$, apply Proposition~\ref{prop:duality} and
	Theorem~\ref{thm:estimate_Besov_multilinear} with
	\begin{equation*}
		r=p',
		\qquad
		t=q'.
	\end{equation*}
	If $q=1$, the conclusion follows directly from the critical embedding
	\begin{equation*}
		\dot B^{\frac np}_{p,1}\subset L^\infty
	\end{equation*}
	by taking $u=v$.
\end{proof}

This proves Theorem~\ref{thm:intro:besov}.

\subsection{Triebel--Lizorkin estimates}

In the Triebel--Lizorkin scale, the outer spatial power must also be expanded.
This accounts for the second arithmetic condition in the theorem below.

\begin{thm}[Multilinear Triebel--Lizorkin estimate]\label{thm:estimate_TL_multilinear}
	Let $n\geq2$, $1\leq l\leq n-1$, and
	\begin{equation*}
		2\leq t\leq r<\infty.
	\end{equation*}
	Assume that $t$ is an even integer and that
	\begin{equation*}
		N=\frac rt\in\mathbb N.
	\end{equation*}
	Whenever $\omega\in L^1_0\Lambda^l$ and $P\omega\in\dot F^{s_r}_{r,t}\Lambda^l$, one has $\omega\in\dot F^{s_r}_{r,t}\Lambda^l$ and
	\begin{equation*}
		\|\omega\|_{\dot F^{s_r}_{r,t}\Lambda^l}
		\lesssim
		\|\omega\|_{L^1_0\Lambda^l}
		+
		\|P\omega\|_{\dot F^{s_r}_{r,t}\Lambda^l}.
	\end{equation*}
\end{thm}

\begin{proof}
	As before, we use the annular regularization from
	Section~\ref{section:hilbertian} to reduce the proof to a form with compact Fourier support away from the origin.
	
	Fix a component $I$ and write
	\begin{equation*}
		u=\omega_I,
		\qquad
		v=(P\omega)_I.
	\end{equation*}
	Since $t$ is even, the function
	\begin{equation*}
		A(x)
		=
		\sum_{j\in\mathbb Z}
		2^{js_rt}(\Delta_ju(x))^t
	\end{equation*}
	is nonnegative, and
	\begin{equation*}
		\|u\|_{\dot F^{s_r}_{r,t}}^r
		=
		\int_{\mathbb R^n}A(x)^Ndx.
	\end{equation*}
	Expanding the integer power and ordering the indices gives
	\begin{equation*}
		\begin{aligned}
		\|u\|_{\dot F^{s_r}_{r,t}}^r
		\sim
		\sum_{j_1\geq\cdots\geq j_N}
		2^{s_rt(j_1+\cdots+j_N)}
		\int_{\mathbb R^n}
		\prod_{k=1}^N(\Delta_{j_k}u)^t dx.
		\end{aligned}
	\end{equation*}
	
	Then, using \eqref{eq:recover_identity_from_S} in one copy of $\Delta_{j_1}u$, we obtain
	\begin{equation}\label{eq:TL_split}
		\|u\|_{\dot F^{s_r}_{r,t}}^r
		\lesssim
		\mathcal E_I+\mathcal P_I,
	\end{equation}
	where
	\begin{equation*}
		\mathcal E_I
		=
		\sum_{j_1\geq\cdots\geq j_N}
		2^{s_rt(j_1+\cdots+j_N)}
		\bigg|
		\int_{\mathbb R^n}
		\Delta_{j_1}(S\omega)_I
		(\Delta_{j_1}u)^{t-1}
		\prod_{k=2}^N(\Delta_{j_k}u)^t dx
		\bigg|
	\end{equation*}
	and
	\begin{equation*}
		\mathcal P_I
		=
		\sum_{j_1,\ldots, j_N}
		2^{s_rt(j_1+\cdots+j_N)}
		\bigg|
		\int_{\mathbb R^n}
		\Delta_{j_1}v
		(\Delta_{j_1}u)^{t-1}
		\prod_{k=2}^N(\Delta_{j_k}u)^t dx
		\bigg|.
	\end{equation*}
	Let
	\begin{equation*}
		F_u(x)
		=
		\bigg(
		\sum_j2^{js_rt}|\Delta_ju(x)|^t
		\bigg)^{1/t}
	\end{equation*}
	and define $F_v$ analogously. Using H\"older's inequality in the sequence variable, we find that
	\begin{equation*}
		\begin{aligned}
		\mathcal P_I
		&\leq
		\int_{\mathbb R^n}
		\bigg(
		\sum_j2^{js_rt}
		|\Delta_jv||\Delta_ju|^{t-1}
		\bigg)
		F_u^{r-t}dx
		\\
		&\leq
		\int_{\mathbb R^n}F_vF_u^{r-1}dx
		\\
		&\leq
		\|v\|_{\dot F^{s_r}_{r,t}}
		\|u\|_{\dot F^{s_r}_{r,t}}^{r-1}.
		\end{aligned}
	\end{equation*}
	
	Now, writing $L=r-1=tN-1$ and noting that
	\begin{equation*}
		s_rt
		=-\frac{nL}{N},
	\end{equation*}
	we apply Lemma~\ref{lemma:multilinear_reflection} with $M=t$ and
	$K=N$ to the scalar multiplier components of $S$ to deduce that
	\begin{equation*}
		\mathcal E_I
		\lesssim
		\|\omega\|_{L^1_0\Lambda^l}^r.
	\end{equation*}
	
	Finally, substitution of the preceding estimates into \eqref{eq:TL_split} yields
	\begin{equation*}
		\|u\|_{\dot F^{s_r}_{r,t}}^r
		\lesssim
		\|\omega\|_{L^1_0\Lambda^l}^r
		+
		\|v\|_{\dot F^{s_r}_{r,t}}
		\|u\|_{\dot F^{s_r}_{r,t}}^{r-1}.
	\end{equation*}
	Young's inequality and summation over the components complete the proof.
\end{proof}

Combining the theorem with the duality principle gives the main
Triebel--Lizorkin result.

\begin{cor}[Triebel--Lizorkin Bourgain--Brezis estimates]\label{cor:TL:main}
	Let $n\geq2$, $1\leq l\leq n-1$, and $1<p\leq q\leq2$. Assume that $q'$ is an even integer and that
	\begin{equation*}
		\frac{p'}{q'}\in\mathbb N.
	\end{equation*}
	Then, for every $v\in\dot F^{\frac np}_{p,q}\Lambda^l$, there exists $u\in(\dot F^{\frac np}_{p,q}\cap L^\infty)\Lambda^l$ such that $du=dv$ and
	\begin{equation*}
		\|u\|_{(\dot F^{\frac np}_{p,q}\cap L^\infty)\Lambda^l}
		\lesssim
		\|v\|_{\dot F^{\frac np}_{p,q}\Lambda^l}.
	\end{equation*}
	In particular,
	\begin{equation*}
		d\big[\dot F^{\frac np}_{p,q}\Lambda^l\big]
		=
		d\big[(\dot F^{\frac np}_{p,q}\cap L^\infty)\Lambda^l\big]
	\end{equation*}
	with equivalent norms.
\end{cor}

\begin{proof}
	Apply Theorem~\ref{thm:estimate_TL_multilinear} with
	\begin{equation*}
		r=p',
		\qquad
		t=q'.
	\end{equation*}
	The Bourgain--Brezis selection estimate then follows by duality directly from Proposition~\ref{prop:duality}.
\end{proof}

This proves Theorem~\ref{thm:2}. In particular, taking $q=2$ gives the
Sobolev Bourgain--Brezis estimate \eqref{conjecture:1} corresponding to
\begin{equation*}
	p=\frac{2k}{2k-1},
	\qquad
	k=1,2,\ldots
\end{equation*}
for every dimension and every form degree.


\section{Consequences for Hodge decompositions}
\label{section:hodge_consequences}

The Hodge projections are not bounded on $L^1$. The Hardy space $\mathscr H^1$, on which the Riesz transforms are bounded, provides the classical endpoint substitute. The estimates proved above give a different replacement. They control Hodge components in critical negative-smoothness spaces under assumptions that allow genuinely integrable data outside $\mathscr H^1$.

\subsection{Hodge estimates beyond the Hardy space}

We first record the most general forms of the dual Triebel--Lizorkin estimates. This provides us with Hodge estimates for integrable data outside the Hardy space $\mathscr H^1$.

For $1\leq r\leq \infty$, we recall that the regularity parameter $s_r$ is defined in \eqref{notation:regularity}.

\begin{cor}[Triebel--Lizorkin Hodge estimates]\label{cor:hodge_sum_TL}
	Let $n\geq2$, $1\leq l\leq n-1$, and $1<r,t<\infty$. Assume that either
	\begin{equation*}
		n<r'+l,
	\end{equation*}
	or that $t$ is an even integer and $r/t\in\mathbb N$. Whenever
	\begin{equation*}
		\omega\in
		\big(L^1_0+\dot F^{s_r}_{r,t}\big)\Lambda^l
		\qquad\text{and}\qquad
		P\omega\in\dot F^{s_r}_{r,t}\Lambda^l,
	\end{equation*}
	one has $\omega\in\dot F^{s_r}_{r,t}\Lambda^l$ and
	\begin{equation*}
		\|\omega\|_{\dot F^{s_r}_{r,t}\Lambda^l}
		\lesssim
		\|\omega\|_{(L^1_0+\dot F^{s_r}_{r,t})\Lambda^l}
		+
		\|P\omega\|_{\dot F^{s_r}_{r,t}\Lambda^l}.
	\end{equation*}
\end{cor}

\begin{proof}
	In the first range, apply Theorem~\ref{thm:1} with the conjugate exponents $r'$ and $t'$, and then use assertion \ref{estimate:equiv2} of Proposition~\ref{prop:duality}. In the second range, the same conclusion follows from Theorem~\ref{thm:estimate_TL_multilinear} and Proposition~\ref{prop:duality}.
\end{proof}

There is also a Besov counterpart to the preceding Hodge estimates.

\begin{cor}[Besov Hodge estimates]
\label{cor:hodge_sum_Besov}
	Let $n\geq2$, $1\leq l\leq n-1$, and
	\begin{equation*}
		2\leq r\leq t\leq\infty.
	\end{equation*}
	Assume that $r$ is an even integer. Whenever
	\begin{equation*}
		\omega\in
		\big(L^1_0+\dot B^{s_r}_{r,t}\big)\Lambda^l
		\qquad\text{and}\qquad
		P\omega\in\dot B^{s_r}_{r,t}\Lambda^l,
	\end{equation*}
	one has $\omega\in\dot B^{s_r}_{r,t}\Lambda^l$ and
	\begin{equation*}
		\|\omega\|_{\dot B^{s_r}_{r,t}\Lambda^l}
		\lesssim
		\|\omega\|_{(L^1_0+\dot B^{s_r}_{r,t})\Lambda^l}
		+
		\|P\omega\|_{\dot B^{s_r}_{r,t}\Lambda^l}.
	\end{equation*}
\end{cor}

\begin{proof}
	If $t=\infty$, the estimate is a direct consequence of the embedding $L^1\subset \dot B^{s_r}_{r,\infty}$. If $t<\infty$, the conclusion follows from Theorem~\ref{thm:estimate_Besov_multilinear} and assertion \ref{estimate:equiv2} of Proposition~\ref{prop:duality}.
\end{proof}

These statements should be compared with the Hardy-space Hodge
decomposition. The inclusions
\begin{equation*}
	\mathscr H^1=\dot F^0_{1,2}
	\subset
	L^1\cap \dot F^{s_r}_{r,t}\subset L^1 + \dot F^{s_r}_{r,t}
\end{equation*}
hold throughout the critical scale. We refer to \cite{jawerth:1977,triebel:1983} for details on such critical embeddings. It is thus noteworthy that Corollaries~\ref{cor:hodge_sum_TL} and~\ref{cor:hodge_sum_Besov} do not require the
integrable part of the data to have integrable Riesz transforms.

\subsection{Hodge--Sobolev estimates}

The preceding results yield estimates directly in terms of $d\omega$ and
$d^*\omega$. The following is formulated in the spirit of Corollary~17 in \cite{bourgain:brezis:2007}, and the somewhat weaker results from \cite{lanzani:stein:2005}. It also extends some of the results from \cite{vanschaftingen:2010,vanschaftingen:2013} on Hodge--Sobolev embeddings.

\begin{cor}[Triebel--Lizorkin Hodge--Sobolev estimate]\label{cor:hodge_sobolev_TL}
	Let $n\geq4$, $2\leq l\leq n-2$, and $1<r,t<\infty$. Assume that either
	\begin{equation*}
		\max\{l,n-l\}<r'-1,
	\end{equation*}
	or that $t$ is an even integer and $r/t\in\mathbb N$. If $\omega\in\mathscr S_0'\Lambda^l$ satisfies
	\begin{equation*}
		d\omega\in
		\big(L^1_0+\dot F^{s_r}_{r,t}\big)\Lambda^{l+1},
		\qquad
		d^*\omega\in
		\big(L^1_0+\dot F^{s_r}_{r,t}\big)\Lambda^{l-1},
	\end{equation*}
	then $\omega\in\dot F^{s_r+1}_{r,t}\Lambda^l$ and
	\begin{equation*}
		\|\omega\|_{\dot F^{s_r+1}_{r,t}\Lambda^l}
		\lesssim
		\|d\omega\|_{(L^1_0+\dot F^{s_r}_{r,t})\Lambda^{l+1}}
		+
		\|d^*\omega\|_{(L^1_0+\dot F^{s_r}_{r,t})\Lambda^{l-1}}.
	\end{equation*}
\end{cor}

\begin{proof}
	The form $d^*\omega$ is coexact. Corollary~\ref{cor:hodge_sum_TL}, applied in degree $l-1$, gives
	\begin{equation*}
		\|d^*\omega\|_{\dot F^{s_r}_{r,t}\Lambda^{l-1}}
		\lesssim
		\|d^*\omega\|_{(L^1_0+\dot F^{s_r}_{r,t})\Lambda^{l-1}}.
	\end{equation*}
	Likewise, $\star d\omega$ is coexact of degree $n-l-1$, and hence
	\begin{equation*}
		\|d\omega\|_{\dot F^{s_r}_{r,t}\Lambda^{l+1}}
		\lesssim
		\|d\omega\|_{(L^1_0+\dot F^{s_r}_{r,t})\Lambda^{l+1}}.
	\end{equation*}
	In the first parameter range, the two applications are permitted by $r'>n-l+1$ and $r'>l+1$, respectively. The second range is independent of the form degree.
	
	Finally, since
	\begin{equation*}
		P\omega=d|D|^{-2}(d^*\omega),
		\qquad
		P^\perp\omega=d^*|D|^{-2}(d\omega),
	\end{equation*}
	the boundedness of order $-1$ of the operators $d|D|^{-2}$ and $d^*|D|^{-2}$ gives the conclusion.
\end{proof}

Here is the Besov version of the previous Hodge--Sobolev estimates.

\begin{cor}[Besov Hodge--Sobolev estimate]
\label{cor:hodge_sobolev_Besov}
	Let $n\geq4$, $2\leq l\leq n-2$, and
	\begin{equation*}
		2\leq r\leq t\leq\infty.
	\end{equation*}
	Assume that $r$ is an even integer. If $\omega\in\mathscr S_0'\Lambda^l$ satisfies
	\begin{equation*}
		d\omega\in
		\big(L^1_0+\dot B^{s_r}_{r,t}\big)\Lambda^{l+1},
		\qquad
		d^*\omega\in
		\big(L^1_0+\dot B^{s_r}_{r,t}\big)\Lambda^{l-1},
	\end{equation*}
	then $\omega\in\dot B^{s_r+1}_{r,t}\Lambda^l$ and
	\begin{equation*}
		\|\omega\|_{\dot B^{s_r+1}_{r,t}\Lambda^l}
		\lesssim
		\|d\omega\|_{(L^1_0+\dot B^{s_r}_{r,t})\Lambda^{l+1}}
		+
		\|d^*\omega\|_{(L^1_0+\dot B^{s_r}_{r,t})\Lambda^{l-1}}.
	\end{equation*}
\end{cor}

\begin{proof}
	It suffices to apply Corollary~\ref{cor:hodge_sum_Besov} to $d^*\omega$ and to $\star d\omega$, and then argue as in the proof of Corollary~\ref{cor:hodge_sobolev_TL}.
\end{proof}

\begin{rem}
	The same argument, without taking one derivative, gives an estimate in terms of the exact and coexact components themselves. More precisely, in the Triebel--Lizorkin scale,
	\begin{equation*}
		\|\omega\|_{\dot F^{s_r}_{r,t}\Lambda^l}
		\lesssim
		\|P\omega\|_{(L^1_0+\dot F^{s_r}_{r,t})\Lambda^l}
		+\|P^\perp\omega\|_{(L^1_0+\dot F^{s_r}_{r,t})\Lambda^l}
	\end{equation*}
	whenever either $\max\{l,n-l\}<r'$ or the arithmetic assumptions of Theorem~\ref{thm:estimate_TL_multilinear} hold. The advantage is that this estimate now holds for $n\geq 2$ and $1\leq l\leq n-1$. An analogous estimate in the Besov scale follows from Theorem~\ref{thm:estimate_Besov_multilinear}.
\end{rem}

\subsection{Critical decompositions with bounded potentials}

On the primal side, the Bourgain--Brezis estimates strengthen the usual critical Hodge decomposition by allowing both potentials to be chosen bounded. The following is formulated in the spirit of Corollary~3 in \cite{bourgain:brezis:2004} and Corollaries~13 and~16 in \cite{bourgain:brezis:2007}.

\begin{cor}[Triebel--Lizorkin Hodge decomposition with bounded potentials]\label{cor:bounded_hodge_TL}
	Let $n\geq4$, $2\leq l\leq n-2$, and $1<p,q<\infty$. Assume that either
	\begin{equation*}
		\max\{l,n-l\}<p-1,
	\end{equation*}
	or that
	\begin{equation*}
		1<p\leq q\leq2,
		\qquad
		q'\text{ is even},
		\qquad
		\frac{p'}{q'}\in\mathbb N.
	\end{equation*}
	Then
	\begin{equation*}
		\dot F^{\frac np-1}_{p,q}\Lambda^l
		=
		d\big[(\dot F^{\frac np}_{p,q}\cap L^\infty)\Lambda^{l-1}\big]
		\oplus
		d^*\big[(\dot F^{\frac np}_{p,q}\cap L^\infty)\Lambda^{l+1}\big]
	\end{equation*}
	with equivalent norms.
\end{cor}

\begin{proof}
	The usual Hodge decomposition gives
	\begin{equation*}
		\dot F^{\frac np-1}_{p,q}\Lambda^l
		=
		d\big[\dot F^{\frac np}_{p,q}\Lambda^{l-1}\big]
		\oplus
		d^*\big[\dot F^{\frac np}_{p,q}\Lambda^{l+1}\big].
	\end{equation*}
	In the first parameter range, Theorem~\ref{thm:1} applies in degrees $l-1$ and $n-l-1$, since
	\begin{equation*}
		p>n-l+1
		\qquad\text{and}\qquad
		p>l+1.
	\end{equation*}
	In the second range, Theorem~\ref{thm:2} applies in every form degree. Applying the corresponding Bourgain--Brezis selection to the exact component and, after conjugation by the Hodge star operator, to the coexact component proves the result.
\end{proof}

The Besov counterpart of the preceding Hodge decomposition also holds.

\begin{cor}[Besov Hodge decomposition with bounded potentials]
\label{cor:bounded_hodge_Besov}
	Let $n\geq4$, $2\leq l\leq n-2$, $1<p\leq2$, and
	$1\leq q\leq p$. Assume that $p'$ is an even integer. Then
	\begin{equation*}
		\dot B^{\frac np-1}_{p,q}\Lambda^l
		=
		d\big[(\dot B^{\frac np}_{p,q}\cap L^\infty)\Lambda^{l-1}\big]
		\oplus
		d^*\big[(\dot B^{\frac np}_{p,q}\cap L^\infty)\Lambda^{l+1}\big]
	\end{equation*}
	with equivalent norms.
\end{cor}

\begin{proof}
	The strategy is similar to the previous proof. It suffices to combine the usual Besov Hodge decomposition with Theorem~\ref{thm:intro:besov}, in degrees $l-1$ and $n-l-1$.
\end{proof}

\begin{rem}
	The same argument one degree higher gives decompositions of the critical spaces themselves through projections of bounded forms. More specifically, in
	the Triebel--Lizorkin scale,
	\begin{equation*}
		\dot F^{\frac np}_{p,q}\Lambda^l
		=
		P\big[(\dot F^{\frac np}_{p,q}\cap L^\infty)\Lambda^l\big]
		\oplus
		P^\perp\big[(\dot F^{\frac np}_{p,q}\cap L^\infty)\Lambda^l\big]
	\end{equation*}
	whenever either $\max\{l,n-l\}<p$ or the arithmetic assumptions of
	Theorem~\ref{thm:2} hold. An advantage is that this decomposition now holds for $n\geq 2$ and $1\leq l\leq n-1$. An analogous Besov statement follows from Theorem~\ref{thm:intro:besov}.
\end{rem}

\subsection{An endpoint Riesz-potential estimate}

We conclude with a direct consequence of the identity \eqref{eq:recover_identity_from_S}, providing an endpoint version of the dual Bourgain--Brezis estimates.

\begin{cor}\label{cor:endpoint_estimate_L_infty}
	Let $n\geq2$ and $1\leq l\leq n-1$. For every finite
	$\Lambda^l$-valued Radon measure $\mu$ such that
	$|D|^{-n}P\mu\in L^\infty\Lambda^l$, one has
	\begin{equation*}
		\big\||D|^{-n}\mu\big\|_{L^\infty\Lambda^l}
		\lesssim
		\|\mu\|_{\mathscr M\Lambda^l}
		+
		\big\||D|^{-n}P\mu\big\|_{L^\infty\Lambda^l}.
	\end{equation*}
	If $\omega\in L^1\Lambda^l$ and
	$|D|^{-n}P\omega\in C\Lambda^l$, then
	\begin{equation*}
		\big\||D|^{-n}\omega\big\|_{C\Lambda^l}
		\lesssim
		\|\omega\|_{L^1\Lambda^l}
		+
		\big\||D|^{-n}P\omega\big\|_{C\Lambda^l}.
	\end{equation*}
	Finally, if $\omega\in L^1_0\Lambda^l$ and
	$|D|^{-n}P\omega\in C_0\Lambda^l$, then
	\begin{equation*}
		\big\||D|^{-n}\omega\big\|_{C_0\Lambda^l}
		\lesssim
		\|\omega\|_{L^1_0\Lambda^l}
		+
		\big\||D|^{-n}P\omega\big\|_{C_0\Lambda^l}.
	\end{equation*}
\end{cor}

\begin{proof}
	By \eqref{eq:recover_identity_from_S},
	\begin{equation*}
		|D|^{-n}\mu
		=
		\frac nl\left(
		|D|^{-n}P\mu-|D|^{-n}S\mu
		\right).
	\end{equation*}
	The three conclusions follow from
	Proposition~\ref{prop:Beurling_transform}.
\end{proof}


\section{The intrinsic nonlinearity of Bourgain--Brezis selections}
\label{section:nonlinearity}

The estimates proved above provide bounded selections for underdetermined Hodge systems. These are linear differential constraints. Such selections cannot, however, be chosen linearly. Bourgain and Brezis established such an obstruction for Hodge systems of degree $1$ or $n-1$ on the torus, which includes the divergence equation. See \cite[Proposition~2]{bourgain:brezis:2003} and \cite[Proposition~9]{bourgain:brezis:2007}. We record here a Euclidean version which applies to arbitrary nonzero homogeneous Fourier multipliers and to the critical Sobolev, Triebel--Lizorkin, and Besov scales.

Let $V$ and $W$ be finite-dimensional real vector spaces. A symbol
\begin{equation*}
	M\in C^\infty\big(
	\mathbb R^n\setminus\{0\},\mathcal L(V,W)
	\big)
\end{equation*}
is said to be homogeneous of degree $\alpha\in\mathbb R$ if
\begin{equation*}
	M(\lambda\xi)=\lambda^\alpha M(\xi)
\end{equation*}
for every $\lambda>0$ and every $\xi\neq0$. The associated multiplier
$M(D)$ acts naturally on $\mathscr S_0'(\mathbb R^n,V)$.

\begin{thm}[Absence of bounded linear selections]
\label{thm:no_linear_selection}
	Let $M(D)$ be a nonzero homogeneous Fourier multiplier as above. Let $X$
	be one of the critical spaces
	\begin{equation*}
		\dot W^{\frac np,p}(\mathbb R^n,V),
		\qquad
		\dot F^{\frac np}_{p,q}(\mathbb R^n,V),
		\qquad
		\dot B^{\frac np}_{p,q}(\mathbb R^n,V),
	\end{equation*}
	where $1<p<\infty$, and where $1<q<\infty$ in the
	Triebel--Lizorkin and Besov cases. Then there is no bounded linear map
	\begin{equation*}
		K:X\rightarrow L^\infty(\mathbb R^n,V)
	\end{equation*}
	such that
	\begin{equation}\label{eq:linear_selection_constraint}
		M(D)Ku=M(D)u
	\end{equation}
	for every $u\in X$.
\end{thm}

\begin{proof}
	We first reduce the Triebel--Lizorkin and Besov cases to the Sobolev case. Choose
	\begin{equation*}
		1<p_0<\min\{p,q\}.
	\end{equation*}
	The standard critical embeddings give
	\begin{equation*}
		\dot W^{\frac n{p_0},p_0}
		\subset
		\dot F^{\frac np}_{p,q}
		\qquad\text{and}\qquad
		\dot W^{\frac n{p_0},p_0}
		\subset
		\dot B^{\frac np}_{p,q}.
	\end{equation*}
	Thus, a bounded map $K$ on either of the latter spaces restricts to a bounded map on a critical Sobolev space and still satisfies \eqref{eq:linear_selection_constraint}. It is therefore enough to treat the Sobolev case
	\begin{equation*}
		X=\dot W^{\frac np,p}(\mathbb R^n,V).
	\end{equation*}
	
	Assume, toward a contradiction, that such an operator $K$ exists. For $h\in\mathbb R^n$, let
	\begin{equation*}
		\tau_hu(x)=u(x-h)
	\end{equation*}
	and, for $R>0$, define the average operator
	\begin{equation*}
		K_Ru
		=\frac1{|B(0,R)|}
		\int_{B(0,R)}
		\tau_hK\tau_{-h}udh.
	\end{equation*}
	Note that this is not understood as a Bochner integral, because translations are not strongly continuous on $L^\infty$ and the integrand is therefore not necessarily strongly measurable. Instead, the element $K_Ru$ is defined as a weak-star integral. That is to say, it is the unique element $K_Ru\in L^\infty$ such that
	\begin{equation*}
		\langle f,K_Ru\rangle
		=\frac1{|B(0,R)|}
		\int_{B(0,R)}
		\langle f,\tau_hK\tau_{-h}u\rangle dh
	\end{equation*}
	for all $f\in L^1_0$. The last integral is well defined, for
	\begin{equation*}
		h\mapsto\langle f,\tau_hK\tau_{-h}u\rangle=\langle \tau_{-h}f,K\tau_{-h}u\rangle
	\end{equation*}
	is continuous.
	
	Now, the family $K_R:X\to L^\infty$ is uniformly bounded. Since $M(D)$ is translation invariant, one also has
	\begin{equation*}
		M(D)K_Ru=M(D)u.
	\end{equation*}
	
	For any fixed $u\in X$, by the Banach--Alaoglu theorem, there exists a sequence $R_k\to\infty$, such that $K_{R_k}u$ converges, as $k\to\infty$, in the weak-star topology of $L^\infty$. Now, since $X$ is separable, a diagonal process shows the existence of a single sequence $R_k\to\infty$ such that $K_{R_k}u$ converges in the same weak-star topology, for all $u\in X$. This allows us to define a bounded linear operator
	\begin{equation*}
		K':X\to L^\infty,
	\end{equation*}
	with
	\begin{equation*}
		M(D) K'u=M(D)u,
	\end{equation*}
	by the weak-star limit
	\begin{equation*}
		K_{R_k}u\rightharpoonup^*K'u
	\end{equation*}
	for every $u\in X$.
	
	Next, we show that $K'$ is translation invariant. To that end, we take $x_0\in\mathbb{R}^n$ and derive the commutator estimate, for any $f\in L^1_0$,
	\begin{equation*}
		\begin{aligned}
			\big|\big\langle f,[\tau_{x_0},K_R]u\big\rangle\big|
			&=\frac 1{|B(0,R)|}\left|\int
			\big(\mathds{1}_{h\in B(x_0,R)}-\mathds{1}_{h\in B(0,R)}\big)
			\big\langle f,\tau_hK\tau_{x_0-h}u\big\rangle dh\right|
			\\
			&\lesssim \|f\|_{L^1_0}\|u\|_X
			\int
			\big|\mathds{1}_{h\in B(R^{-1}x_0,1)}-\mathds{1}_{h\in B(0,1)}\big|
			dh
			\lesssim \frac{|x_0|}{R}\|f\|_{L^1_0}\|u\|_X.
		\end{aligned}
	\end{equation*}
	Therefore, by weak-star convergence of $K_{R_k}u$, we infer that
	\begin{equation*}
		[\tau_{x_0},K']=0,
	\end{equation*}
	which establishes translation invariance of $K'$.
	
	We next average over dilations. For $\lambda>0$, set
	\begin{equation*}
		\delta_\lambda u(x)=u\Big(\frac x\lambda\Big).
	\end{equation*}
	Since the Sobolev norm is critical, $\delta_\lambda$ acts isometrically on both $X$ and $L^\infty$. For $A>1$, define
	\begin{equation*}
		K_A'u
		=
		\frac1{2\log A}
		\int_{A^{-1}}^A
		\delta_\lambda K'\delta_{\lambda^{-1}}u
		\frac{d\lambda}{\lambda},
	\end{equation*}
	in the sense of a weak-star integral, as before.
	Here, again, the family $K_A':X\to L^\infty$ is uniformly bounded, and the homogeneity of $M$ implies
	\begin{equation*}
		M(D)K_A'u=M(D)u.
	\end{equation*}
	Moreover, noticing that $\tau_h\delta_\lambda=\delta_\lambda\tau_{\lambda^{-1}h}$ and $\tau_{\lambda^{-1}h}\delta_{\lambda^{-1}}=\delta_{\lambda^{-1}}\tau_h$, we see that $K'_A$ is translation invariant.

	Thus, a similar convergence argument in the weak-star topology of $L^\infty$ allows us to define a bounded linear operator
	\begin{equation*}
		K'':X\to L^\infty,
	\end{equation*}
	by a weak-star limit
	\begin{equation*}
		K_{A_k}'u\rightharpoonup^* K''u,
	\end{equation*}
	for some suitable sequence $A_k\to\infty$. It follows that $K''$ is translation invariant and
	\begin{equation}\label{system:hodge}
		M(D)K''u=M(D)u
	\end{equation}
	for all $u\in X$.
	
	We show now that $K''$ is homogeneous of order zero. To that end, considering any $\mu>0$, we find, for any $f\in L^1_0$, that
	\begin{equation*}
		\begin{aligned}
			\big|\big\langle f,[\delta_\mu,K_A']u\big\rangle\big|
			&=\frac 1{2\log A}\left|\int
			\big(\mathds{1}_{\lambda\in [\mu A^{-1},\mu A]}-\mathds{1}_{\lambda\in [A^{-1},A]}\big)
			\big\langle f,\delta_\lambda K'\delta_{\mu\lambda^{-1}}u\big\rangle \frac{d\lambda}\lambda\right|
			\\
			&\lesssim \|f\|_{L^1_0}\|u\|_{X}
			\frac 1{\log A}\int
			\big|\mathds{1}_{\lambda\in [\mu A^{-1},\mu A]}-\mathds{1}_{\lambda\in [A^{-1},A]}\big|
			\frac{d\lambda}\lambda
			\\
			&= \|f\|_{L^1_0}\|u\|_{X}
			\int
			\big|\mathds{1}_{s\in \left[\frac{\log \mu}{\log A}-1,\frac{\log \mu}{\log A} +1\right]}-\mathds{1}_{s\in [-1,1]}\big|
			ds
			\lesssim \frac{|\log\mu|}{\log A}\|f\|_{L^1_0}\|u\|_{X}.
		\end{aligned}
	\end{equation*}
	Therefore, by weak-star convergence of $K_{A_k}'u$, we deduce that
	\begin{equation*}
		[\delta_\mu,K'']=0,
	\end{equation*}
	which establishes the homogeneity of $K''$.
	
	Consequently, by translation invariance of the bounded operator
	\begin{equation*}
		K''|D|^{-\frac np}:L^p\to L^\infty,
	\end{equation*}
	we conclude that it can be written as a convolution
	\begin{equation*}
		K''|D|^{-\frac np}u=\int_{\mathbb{R}^n}k(x-y)u(y)dy,
	\end{equation*}
	defined in the homogeneous sense on subspaces of $\mathscr S_0'$, where $k(x)$ is a linear map in $\mathcal{L}(V,V)$ with coefficients in $L^{p'}(\mathbb{R}^n)$. We refer to the seminal work of H\"ormander \cite{hormander:1960} for details on translation invariant operators.
	
	On the other hand, exploiting the homogeneity of $K''$, we deduce the homogeneity of the kernel
	\begin{equation*}
		k(x)=\lambda^{n(1-\frac 1p)}k(\lambda x),
	\end{equation*}
	for all $\lambda>0$. For every $j\in\mathbb Z$, this gives
	\begin{equation*}
		\int_{\mathbb{R}^n}|k(x)|^{p'}dx
		=\sum_{j\in\mathbb{Z}}
		\int_{\{2^j\leq |x|<2^{j+1}\}}
		|k(x)|^{p'}dx
		=\sum_{j\in\mathbb{Z}}\int_{\{1\leq |x|<2\}}
		|k(x)|^{p'}dx,
	\end{equation*}
	which implies that $k=0$, because $k\in L^{p'}$.
	
	Since $|D|^{-\frac np}$ maps $L^p$ onto $X$, it follows that $K''=0$, which, in view of \eqref{system:hodge}, gives
	\begin{equation*}
		M(D)u=0
	\end{equation*}
	for every $u\in X$. This contradicts the assumption that $M$ is nonzero. Indeed, one may choose $u\in\mathscr S_0(\mathbb R^n,V)$ whose Fourier transform is supported near a point at which $M$ does not vanish.
	
	We conclude that the linear operator $K$ does not exist, thereby completing the proof.
\end{proof}

Taking $V=\Lambda^l$, $W=\Lambda^{l+1}$, and $M(D)=d$ gives the conclusion relevant to the present paper.

\begin{cor}[Nonlinearity of Bourgain--Brezis selections]\label{cor:nonlinear_hodge_selection}
	Let $n\geq2$, $1\leq l\leq n-1$, and let $X$ be any of the critical
	spaces appearing in Theorem~\ref{thm:no_linear_selection}. There is no
	bounded linear map
	\begin{equation*}
		K:X\Lambda^l\rightarrow L^\infty\Lambda^l
	\end{equation*}
	such that
	\begin{equation*}
		dKv=dv
	\end{equation*}
	for every $v\in X\Lambda^l$.
	
	In particular, except for the endpoint Besov case $q=1$, where the critical space already embeds into $L^\infty$, the bounded Bourgain--Brezis selections furnished by the main results of this paper are necessarily nonlinear.
\end{cor}

\begin{rem}
	The restriction $q>1$ in the Besov statement is essential. At the endpoint index $q=1$, one has
	\begin{equation*}
		\dot B^{\frac np}_{p,1}\subset L^\infty,
	\end{equation*}
	so the identity map itself is a bounded linear selection.
\end{rem}

\appendix


\section{A one-dimensional model}
\label{appendix:one-dimensional-model}

The core arguments presented in this paper depend on the reflection symmetries of the scalar multipliers appearing in the trace-free Beurling--Ahlfors transform. The vector-valued structure of differential forms is in fact not needed for this basic mechanism. We illustrate this point in one dimension, where the Hodge projections are replaced by the positive- and negative-frequency Hardy projectors
\begin{equation*}
	\Pi_+=\mathds 1_{\{D>0\}},
	\qquad
	\Pi_-=\mathds 1_{\{D<0\}}.
\end{equation*}
All function spaces in this appendix are complex-valued, and all $L^2$ pairings are Hermitian. Let
\begin{equation*}
	H=-i\operatorname{sign}(D)
\end{equation*}
be the Hilbert transform. Then
\begin{equation*}
	\Pi_+=\frac12(\operatorname{Id}+iH),
	\qquad
	\Pi_-=\frac12(\operatorname{Id}-iH),
\end{equation*}
and
\begin{equation}\label{eq:one_dimensional_decomposition}
	\operatorname{Id}
	=
	2\Pi_-+iH.
\end{equation}
The multiplier $\operatorname{sign}(\xi)$ is homogeneous of degree zero and odd under the reflection $\xi\mapsto-\xi$. Hence Lemma~\ref{lemma:bound:oddness} gives
\begin{equation}\label{eq:one_dimensional_endpoint}
	\big\||D|^{-1}\operatorname{sign}(D)\mu\big\|_{L^\infty}
	\lesssim
	\|\mu\|_{\mathscr{M}}
\end{equation}
for every $\mu\in \mathscr M(\mathbb R)$. Moreover, the operator $|D|^{-1}\operatorname{sign}(D)$ maps $L^1$ into $C$ and $L^1_0$ into $C_0$. In this special case, the cancellation is also visible from
\begin{equation*}
	\frac d{dx}|D|^{-1}\operatorname{sign}(D)
	=
	i\operatorname{Id}.
\end{equation*}

\subsection{Hilbertian and Besov estimates}

We first record the Hilbertian one-dimensional estimate.

\begin{prop}[Hilbertian one-dimensional estimate]
\label{prop:toy_hilbertian}
	Let $f\in L^1_0(\mathbb R)$ and assume that $\Pi_-f\in\dot H^{-\frac12}(\mathbb R)$. Then $f\in\dot H^{-\frac12}(\mathbb R)$ and
	\begin{equation*}
		\|f\|_{\dot H^{-\frac12}}
		\lesssim
		\|f\|_{L^1_0}
		+
		\|\Pi_-f\|_{\dot H^{-\frac12}}.
	\end{equation*}
\end{prop}

\begin{proof}
	Using the annular regularization from Section~\ref{section:hilbertian}, we may first suppose that the Fourier transform of $f$ is supported in a compact annulus and $f\in\mathscr S_0$. Since $\Pi_+$ and $\Pi_-$ are orthogonal projections,
	\begin{equation*}
		\|\Pi_+f\|_{\dot H^{-\frac12}}^2
		-
		\|\Pi_-f\|_{\dot H^{-\frac12}}^2
		=
		\big\langle
		f,
		|D|^{-1}\operatorname{sign}(D)f
		\big\rangle.
	\end{equation*}
	Therefore, by \eqref{eq:one_dimensional_endpoint},
	\begin{equation*}
		\|\Pi_+f\|_{\dot H^{-\frac12}}^2
		\lesssim
		\|\Pi_-f\|_{\dot H^{-\frac12}}^2
		+\|f\|_{L^1_0}^2.
	\end{equation*}
	Since
	\begin{equation*}
		\|f\|_{\dot H^{-\frac12}}^2
		=
		\|\Pi_+f\|_{\dot H^{-\frac12}}^2
		+
		\|\Pi_-f\|_{\dot H^{-\frac12}}^2,
	\end{equation*}
	the estimate follows. The general case is obtained by approximation.
\end{proof}

The corresponding selection statement follows from the same Hahn--Banach
argument as Proposition~\ref{prop:duality}, with $\Pi_-$ and $\Pi_+$ in
place of $P$ and $P^\perp$.

\begin{cor}[Hilbertian positive-frequency selection]
\label{cor:toy_hilbertian_selection}
	For every $g\in\dot H^{\frac12}(\mathbb R)$, there exists
	$h\in(\dot H^{\frac12}\cap L^\infty)(\mathbb R)$ such that
	\begin{equation*}
		\Pi_+h=\Pi_+g
	\end{equation*}
	and
	\begin{equation*}
		\|h\|_{\dot H^{\frac12}\cap L^\infty}
		\lesssim
		\|g\|_{\dot H^{\frac12}}.
	\end{equation*}
	In particular,
	\begin{equation*}
		\Pi_+\big[\dot H^{\frac12}(\mathbb R)\big]
		=
		\Pi_+\big[(\dot H^{\frac12}\cap L^\infty)(\mathbb R)\big]
	\end{equation*}
	with equivalent norms.
\end{cor}

There is also a dyadic refinement.

\begin{prop}[Besov one-dimensional estimate]
\label{prop:toy_besov}
	Let $2\leq r\leq\infty$. If $f\in L^1_0(\mathbb R)$ and $\Pi_-f\in\dot B^{-\frac12}_{2,r}(\mathbb R)$, then $f\in\dot B^{-\frac12}_{2,r}(\mathbb R)$ and
	\begin{equation*}
		\|f\|_{\dot B^{-\frac12}_{2,r}}
		\lesssim
		\|f\|_{L^1_0}
		+
		\|\Pi_-f\|_{\dot B^{-\frac12}_{2,r}}.
	\end{equation*}
\end{prop}

\begin{proof}
	We again argue first after annular regularization and assume $f\in\mathscr S_0$. Set
	\begin{equation*}
		a_j
		=
		\|\Delta_j\Pi_+f\|_{\dot H^{-\frac12}},
		\qquad
		b_j
		=
		\|\Delta_j\Pi_-f\|_{\dot H^{-\frac12}}.
	\end{equation*}
	The localized Hilbertian identity gives
	\begin{equation*}
		a_j^2-b_j^2
		=
		\left\langle
		f,\Delta_j^2|D|^{-1}\operatorname{sign}(D)f
		\right\rangle.
	\end{equation*}
	Hence
	\begin{equation*}
		a_j^2
		\leq
		b_j^2+
		\int_{\mathbb R}
		|f|
		\big|\Delta_j^2|D|^{-1}\operatorname{sign}(D)f\big|dx.
	\end{equation*}
	By the dyadic form of the reflection estimate found in Corollary~\ref{scholium:bound_m}, we deduce
	\begin{equation*}
		\sum_{j\in\mathbb Z}
		\int_{\mathbb R}
		|f|
		\big|\Delta_j^2|D|^{-1}\operatorname{sign}(D)f\big|dx
		\leq
		\|f\|_{L^1_0}
		\bigg\|
		\sum_{j\in\mathbb Z}
		\big|\Delta_j^2|D|^{-1}\operatorname{sign}(D)f\big|
		\bigg\|_{C_0}
		\lesssim
		\|f\|_{L^1_0}^2.
	\end{equation*}
	As before, the case $r=\infty$ follows directly from the critical embedding $L^1\subset \dot B^{-\frac 12}_{2,\infty}$.
	If $r<\infty$, then
	\begin{equation*}
		\begin{aligned}
		\|(a_j)\|_{\ell^r}^2
		&=
		\|(a_j^2)\|_{\ell^{r/2}}
		\\
		&\lesssim
		\|(b_j^2)\|_{\ell^{r/2}}
		+
		\|f\|_{L^1_0}^2
		\\
		&\leq
		\|(b_j)\|_{\ell^r}^2
		+
		\|f\|_{L^1_0}^2.
		\end{aligned}
	\end{equation*}
	This proves the result by a density argument.
\end{proof}

\begin{cor}[Besov positive-frequency selection]
\label{cor:toy_besov_selection}
	Let $1\leq q\leq2$. For every $g\in\dot B^{\frac12}_{2,q}(\mathbb R)$, there exists $h\in(\dot B^{\frac12}_{2,q}\cap L^\infty)(\mathbb R)$ such that
	\begin{equation*}
		\Pi_+h=\Pi_+g
	\end{equation*}
	and
	\begin{equation*}
		\|h\|_{\dot B^{\frac12}_{2,q}\cap L^\infty}
		\lesssim
		\|g\|_{\dot B^{\frac12}_{2,q}}.
	\end{equation*}
\end{cor}

\begin{proof}
	If $1<q\leq2$, apply the same duality argument as above to Proposition~\ref{prop:toy_besov}, with $r=q'$. If $q=1$, the conclusion is immediate from the critical embedding $\dot B^{\frac12}_{2,1}\subset L^\infty$, by taking $h=g$.
\end{proof}

\subsection{Multilinear extensions}

The multilinear reflection argument of Section~\ref{section:multilinear} applies directly to the odd multiplier $\operatorname{sign}(D)$. Combining \eqref{eq:one_dimensional_decomposition} with the proofs of the Besov and Triebel--Lizorkin estimates gives the following one-dimensional analogues of the main results.

\begin{cor}[Non-Hilbertian positive-frequency selections]
\label{cor:toy_multilinear}
	The following statements hold.
	\begin{enumerate}
		\item
		Let $1<p\leq2$ and $1\leq q\leq p$, and assume that $p'$ is an
		even integer. Then
		\begin{equation*}
			\Pi_+\big[\dot B^{\frac1p}_{p,q}(\mathbb R)\big]
			=
			\Pi_+\big[(\dot B^{\frac1p}_{p,q}\cap L^\infty)(\mathbb R)\big]
		\end{equation*}
		with equivalent norms.
		
		\item
		Let $1<p\leq q\leq2$, and assume that $q'$ is even and that
		$p'/q'\in\mathbb N$. Then
		\begin{equation*}
			\Pi_+\big[\dot F^{\frac1p}_{p,q}(\mathbb R)\big]
			=
			\Pi_+\big[(\dot F^{\frac1p}_{p,q}\cap L^\infty)(\mathbb R)\big]
		\end{equation*}
		with equivalent norms.
	\end{enumerate}
\end{cor}

\begin{proof}
	The scalar multiplier $\operatorname{sign}(\xi)$ satisfies the reflection hypothesis of Lemma~\ref{lemma:multilinear_reflection}. The arguments in the Besov and Triebel--Lizorkin parts of Section~\ref{section:multilinear} therefore apply verbatim, with $n=1$, the Hardy projector $\Pi_-$ in place of $P$, and \eqref{eq:one_dimensional_decomposition} in place of \eqref{eq:recover_identity_from_S}. The selection statements follow by the same duality argument as previously.
\end{proof}

\subsection{Holomorphic interpretation}

The positive-frequency Hardy projection $\Pi_+$ identifies the boundary values of holomorphic functions on the upper half-plane. More precisely, consider a holomorphic function $F(x+iy)$ in the upper half-plane and assume that its trace is formally well-defined. On the real line, it holds
\begin{equation*}
	F(x)=\frac 12\big(f(x)+iHf(x)\big)=\Pi_+f(x).
\end{equation*}
The choice of $f$ is not unique, since its negative-frequency component does not contribute to $F$. If $f$ is required to be real-valued, however, then it is uniquely determined by $f=2\operatorname{Re}F$.

The Dirichlet space consists of holomorphic functions on the upper half-plane with a square-integrable derivative. Through boundary traces, it is isomorphically identified with
\begin{equation*}
	\Pi_+\big[\dot H^{\frac12}(\mathbb R)\big],
\end{equation*}
with equivalent norms.

Corollary~\ref{cor:toy_hilbertian_selection} gives the alternative representation of the Dirichlet space
\begin{equation*}
	\Pi_+\big[\dot H^{\frac12}(\mathbb R)\big]
	=
	\Pi_+\big[(\dot H^{\frac12}\cap L^\infty)(\mathbb R)\big].
\end{equation*}
Thus every Dirichlet boundary value admits a bounded critical-Sobolev preimage under $\Pi_+$. We refer to \cite{arcozzi:rochberg:sawyer:wick:2019} for background on the Dirichlet space, and to \cite{dalio:riviere:wettstein:2021} for a related Bourgain--Brezis phenomenon in a holomorphic setting.

The same statement yields a refinement of the Fefferman--Stein decomposition on the critical Sobolev subspace.

\begin{cor}\label{cor:toy_fefferman_stein}
	Every real-valued $f\in\dot H^{\frac12}(\mathbb R)$ can be written as
	\begin{equation*}
		f=a-Hb,
	\end{equation*}
	for some real-valued $a,b\in(\dot H^{\frac12}\cap L^\infty)(\mathbb R)$ with
	\begin{equation*}
		\|a\|_{\dot H^{\frac12}\cap L^\infty}
		+
		\|b\|_{\dot H^{\frac12}\cap L^\infty}
		\lesssim
		\|f\|_{\dot H^{\frac12}}.
	\end{equation*}
	Consequently,
	\begin{equation*}
		\dot H^{\frac12}(\mathbb R)
		=
		(\dot H^{\frac12}\cap L^\infty)(\mathbb R)
		+
		H(\dot H^{\frac12}\cap L^\infty)(\mathbb R)
	\end{equation*}
	with equivalent induced norms.
\end{cor}

\begin{proof}
	Apply Corollary~\ref{cor:toy_hilbertian_selection} to $f$ and write the resulting bounded function as $h=a+ib$, with $a$ and $b$ real-valued.
	The identity $\Pi_+f=\Pi_+h$ becomes
	\begin{equation*}
		\frac12(f+iHf)
		=
		\frac12\big(a-Hb+i(b+Ha)\big).
	\end{equation*}
	Taking real parts gives $f=a-Hb$, and the estimate follows from the
	bound on $h$.
\end{proof}

The Fefferman--Stein \cite{fefferman:stein:1972} and Uchiyama \cite{uchiyama:1982} decomposition of functions of bounded mean oscillation states that
\begin{equation*}
	BMO(\mathbb{R})=L^\infty(\mathbb{R}) + HL^\infty(\mathbb{R}).
\end{equation*}
The previous result provides a refinement of this decomposition to the subspace $\dot H^\frac 12\subset BMO$.

As noted in \cite{bourgain:brezis:2003}, a similar application of the Bourgain--Brezis selection principle for Hodge systems leads to subtle improvements of the Fefferman--Stein characterization of $BMO$ in higher dimensions.


\section{Further properties of the trace-free Beurling--Ahlfors transform}
\label{appendix:further_properties}

We conclude by recording two supplementary properties of the trace-free Beurling--Ahlfors transform. The first shows that its endpoint cancellation persists after composition with a Riesz transform. The second describes the density of its endpoint image in $C_0\Lambda^l$.

\subsection{Riesz transforms of the endpoint image}

\begin{prop}[Endpoint bounds for $R_kS$]
\label{prop:riesz_endpoint_S}
	Let $n\geq2$, $1\leq l\leq n-1$, and $1\leq k\leq n$. Then
	\begin{equation*}
		|D|^{-n}R_kS:
		\mathscr M\Lambda^l
		\rightarrow
		L^\infty\Lambda^l
	\end{equation*}
	is bounded. Moreover,
	\begin{equation*}
		|D|^{-n}R_kS[L^1\Lambda^l]\subset C\Lambda^l,
		\qquad
		|D|^{-n}R_kS[L^1_0\Lambda^l]\subset C_0\Lambda^l.
	\end{equation*}
	The corresponding dyadic estimate also holds:
	\begin{equation*}
		\Big\|
		\sum_{j\in\mathbb Z}
		\big|\Delta_j|D|^{-n}R_kS\mu\big|
		\Big\|_{L^\infty}
		\lesssim
		\|\mu\|_{\mathscr M\Lambda^l}.
	\end{equation*}
	If $\mu=\omega$ belongs to $L^1\Lambda^l$ or $L^1_0\Lambda^l$, then the function in the preceding inequality belongs to $C$ or $C_0$, respectively. The same conclusions hold with $\Delta_j^2$ in place of $\Delta_j$.
\end{prop}

\begin{proof}
	By Proposition~\ref{prop:beurling_multiplier_decomposition}, every matrix coefficient of $R_kS$ is a finite linear combination of the scalar multipliers
	\begin{equation*}
		\widetilde m_{ij}^{\sharp}(\xi)
		=
		\frac{i\xi_k(\xi_i^2-\xi_j^2)}{|\xi|^3},
		\qquad
		\widetilde m_{ij}^{\flat}(\xi)
		=
		\frac{i\xi_k\xi_i\xi_j}{|\xi|^3},
		\qquad
		i\neq j.
	\end{equation*}
	Both families are smooth away from the origin and homogeneous of degree zero. The multiplier $\widetilde m_{ij}^{\sharp}$ is odd under the reflection across $e_k^\perp$. For $\widetilde m_{ij}^{\flat}$, if $k\neq i$, reflection across $e_i^\perp$ changes its sign, whereas if $k=i$, reflection across $e_j^\perp$ changes its sign. Hence every scalar multiplier appearing in $R_kS$ satisfies the reflection hypothesis of Lemma~\ref{lemma:bound:oddness}.
	The measure estimate and the continuity conclusions therefore follow coefficientwise from Lemma~\ref{lemma:bound:oddness}. The dyadic statements follow in the same way from Corollary~\ref{scholium:bound_m}.
\end{proof}

In particular, if
\begin{equation*}
	u=|D|^{-n}S\omega,
	\qquad
	\omega\in L^1_0\Lambda^l,
\end{equation*}
then
\begin{equation*}
	u\in C_0\Lambda^l,
	\qquad
	R_ku\in C_0\Lambda^l,
	\quad
	k=1,\ldots,n.
\end{equation*}
Thus the endpoint image has additional Riesz-transform regularity. The next result shows that, despite this restriction, it is dense in $C_0\Lambda^l$.

\subsection{Density of the endpoint image}

\begin{cor}[Density of the range]
\label{cor:dense_range_S}
	The range of
	\begin{equation*}
		|D|^{-n}S:
		L^1_0\Lambda^l
		\rightarrow
		C_0\Lambda^l
	\end{equation*}
	is dense in $C_0\Lambda^l$ with respect to the uniform norm.
\end{cor}

\begin{proof}
	Proposition~\ref{prop:beurling_algebra} shows that $S$ is an automorphism of $\mathscr S_0\Lambda^l$. The multiplier $|D|^{-n}$ is also an automorphism of $\mathscr S_0\Lambda^l$, with inverse $|D|^n$.
	Consequently,
	\begin{equation*}
		|D|^{-n}S[\mathscr S_0\Lambda^l]
		=
		\mathscr S_0\Lambda^l.
	\end{equation*}
	
	It remains to recall that $\mathscr S_0$ is uniformly dense in $C_0$. Indeed, let $g\in\mathscr S$ and choose $\chi\in C_c^\infty(\mathbb R^n)$ such that $\chi=1$ in a neighborhood of the origin. For $\varepsilon>0$, set
	\begin{equation*}
		g_\varepsilon
		=
		\mathscr F^{-1}\left[
		\left(1-\chi\left(\frac{\xi}{\varepsilon}\right)\right)
		\widehat g(\xi)
		\right].
	\end{equation*}
	The Fourier transform of $g_\varepsilon$ vanishes in a neighborhood of the origin, and hence $g_\varepsilon\in\mathscr S_0$. Moreover,
	\begin{equation*}
		\|g-g_\varepsilon\|_{L^\infty}
		\lesssim
		\left\|
		\chi\left(\frac{\cdot}{\varepsilon}\right)\widehat g
		\right\|_{L^1}
		\rightarrow0
	\end{equation*}
	as $\varepsilon\to0$. Since $\mathscr S$ is uniformly dense in $C_0$, it follows that $\mathscr S_0$ is uniformly dense in $C_0$.
	
	Finally, since $\mathscr S_0\Lambda^l\subset L^1_0\Lambda^l$, the range of
	$|D|^{-n}S$ on $L^1_0\Lambda^l$ contains the uniformly dense subspace
	$\mathscr S_0\Lambda^l\subset C_0\Lambda^l$. This proves the result.
\end{proof}


\bibliographystyle{alpha}
\bibliography{hodge}

\end{document}